\documentclass[a4paper, twoside, 10pt]{amsart}

\usepackage[T1]{fontenc}
\usepackage[utf8]{inputenc}
\usepackage[english]{babel}
\usepackage{lmodern}
\usepackage{amssymb,amsmath,amsthm,amsfonts,mathtools}
\RequirePackage{mathrsfs}
\usepackage{tikz-cd, tikz-3dplot}
\usepackage[margin=3cm]{geometry}
\usepackage{enumitem}
\usepackage[hidelinks]{hyperref}
\usepackage{booktabs}
\usepackage[font=small, margin=40pt,figureposition=below]{caption}

\setlist[enumerate, 1]{label = (\roman*), ref = \roman*}

\newcommand*\diff{\mathop{}\!\mathrm{d}}

\newcommand{\pc}[1]{%
  \begingroup\lccode`~=`: \lowercase{\endgroup
  \edef~}{\mathbin{\mathchar\the\mathcode`:}\nobreak}%
  (%
  \begingroup
  \mathcode`:=\string"8000
  #1%
  \endgroup
  )%
}

\newcommand{\uA}{{\boldsymbol{A}}}
\newcommand{\uB}{{\boldsymbol{B}}}

\newcommand{\cA}{\mathcal{A}}
\newcommand{\cB}{\mathcal{B}}
\newcommand{\cC}{\mathcal{C}}
\newcommand{\cD}{\mathcal{D}}
\newcommand{\cE}{\mathcal{E}}
\newcommand{\cK}{\mathcal{K}}
\newcommand{\cL}{\mathcal{L}}
\newcommand{\cO}{\mathcal{O}}

\newcommand{\charfun}{\mathcal{X}}

\newcommand{\mU}{\mathfrak{U}}
\newcommand{\mX}{\mathfrak{X}}
\newcommand{\mY}{\mathfrak{Y}}

\newcommand{\fo}{\mathfrak{o}}

\newcommand{\bAA}{\mathbb{A}}
\newcommand{\CC}{\mathbb{C}}
\newcommand{\PP}{\mathbb{P}}
\newcommand{\RR}{\mathbb{R}}
\newcommand{\ZZ}{\mathbb{Z}}
\newcommand{\ZZnz}{\mathbb{Z}_{\neq{}0}}
\newcommand{\multgrp}{\mathbb{G}_\mathrm{m}}
\newcommand{\GmZ}{\mathbb{G}_{\mathrm{m},\mathbb{Z}}}

\newcommand{\Kbar}{{\overline{K}}}
\newcommand{\QQ}{\mathbb{Q}}

\newcommand{\Can}{\cC^{\mathrm{an}}}
\newcommand{\Canmax}{\cC^{\mathrm{an,max}}}

\DeclareMathOperator{\Cl}{Cl}
\DeclareMathOperator{\CH}{CH}
\DeclareMathOperator{\Div}{Div}
\DeclareMathOperator{\sdiv}{div}
\DeclareMathOperator{\Eff}{Eff}
\DeclareMathOperator{\im}{im}
\DeclareMathOperator{\ord}{ord}
\DeclareMathOperator{\Pic}{Pic}
\DeclareMathOperator{\Proj}{Proj}
\DeclareMathOperator{\Reg}{Reg}
\DeclareMathOperator{\rk}{rk}
\DeclareMathOperator{\Spec}{Spec}
\DeclareMathOperator{\supp}{supp}

\newcommand{\abs}[1]{\left\lvert#1\right\rvert}
\newcommand{\norm}[1]{\left\lVert#1\right\rVert}
\newcommand{\relmiddle}[1]{\mathrel{}\middle#1\mathrel{}} 

\newcommand{\fppf}{\mathrm{fppf}}
\newcommand{\fin}{\mathrm{fin}}
\newcommand{\irr}{\mathrm{irr}}
\newcommand{\tmax}{\mathrm{max}}

\newtheorem{theorem}{Theorem}[subsection]
\newtheorem{corollary}[theorem]{Corollary}
\newtheorem{lemma}[theorem]{Lemma}
\newtheorem{proposition}[theorem]{Proposition}

\theoremstyle{definition}
\newtheorem{definition}[theorem]{Definition}
\newtheorem{example}[theorem]{Example}
\newtheorem{remark}[theorem]{Remark}
\newtheorem{notation}[theorem]{Notation}

\makeatletter
\@ifpackageloaded{hyperref}{
  \DeclareRobustCommand{\SkipTocEntry}[5]{}

}{
  \DeclareRobustCommand{\SkipTocEntry}[4]{}
}
\makeatother

\begin{document}
\title[Integral points of bounded height on a certain toric variety]{Integral points of bounded height\\on a certain toric variety}
\author{Florian Wilsch}
\address{Leibniz Universität Hannover, Institut für Algebra, Zahlentheorie und Diskrete Mathematik, Welfengarten 1, 30167 Hannover}
\email{wilsch@math.uni-hannover.de}
\date{December 8, 2023}

\begin{abstract}
  We determine an asymptotic formula for the number of integral points of bounded height on a certain toric variety, which is incompatible with part of a preprint by Chambert-Loir and Tschinkel.
  We provide an alternative interpretation of the asymptotic formula we get.
  To do so, we construct an analogue of Peyre's constant $\alpha$ and describe its relation to a new obstruction to the Zariski density of integral points in certain regions of varieties. 
\end{abstract}

\maketitle

\tableofcontents

\section{Introduction}

A classical problem in Diophantine geometry is to understand the number of rational points on a given algebraic variety. For Fano varieties---that is, smooth projective varieties whose anticanonical bundle is ample---the set of rational points is expected to be Zariski dense (thus in particular infinite) as soon as it is nonempty. 
A conjecture of Manin's~\cite{MR974910} makes a more precise quantitative prediction for this setting. 
Given a Fano variety $X$ over a number field $K$, one orders the set of rational points $X(K)$ by an \emph{anticanonical height} $H\colon X(K)\to \RR_{>0}$. There might be a closed subvariety $Z\subset X$, or, more generally, a \emph{thin} 
subset $Z\subset X(K)$ of the rational points on $X$ that dominates the number of rational points of bounded height, and one should count points on its complement $V = X(K)\setminus Z$. 
In its current form (cf.\ e.g.~\cite[Conj.~1.2]{MR4472281} and~\cite{MR1032922,MR1340296,MR1679843,MR2019019} for variants and important waypoints leading to its current formulation), the conjecture predicts that the number
\[
  N_{V,H}(B) = \# \{ x \in V \mid  H(x) \le B\}
\]
of rational points of height at most $B$ conforms to the asymptotic formula
\begin{equation}\label{eq:Manin-Peyre}
  N_{V,H}(B) \sim c_{V,H} B (\log B)^{\rk\Pic X-1},
\end{equation}
where the exponent of $\log B$ is the rank of the Picard group of $X$ and $c_{V,H}>0$ is an explicit constant depending in $X$ and $H$.

\medskip

A related problem is the quantitative study of \emph{integral points}. On projective varieties, rational and integral points can be seen to coincide by clearing denominators in the solution to a homogeneous equation or, more formally, as a consequence of the valuative criterion for properness. A problem analogous to the one treated by Manin's conjecture, asking about integral points on a quasiprojective variety, is the following: Let $X$ be a smooth, projective variety over a number field~$K$ and $D$ be a reduced, effective divisor with strict normal crossings such that the log anticanonical bundle $\omega_X(D)^\vee$ is at least big. Let $\mU$ be an integral model of $U=X \setminus D$, and let $H$ be a \emph{log anticanonical height function}. The number of integral points on $\mU$ of bounded height might be dominated by points lying on an \emph{accumulating} thin set $Z\subset X(K)$, which should be excluded. What is the asymptotic behavior of the number
\begin{equation}\label{eq:integral-count-general}
  \#\{x\in \mU(\fo_K) \cap V \mid H(x)\le B\}
\end{equation}
of integral points of bounded height that (as rational points) belong to the complement $V = X(K)\setminus Z$ of accumulating subsets? 

For complete intersections of low degree compared to their dimension, this kind of problem can be studied using the circle method (e.g.~\cite{MR0150129,MR781588}). Methods such as harmonic analysis exploiting a group action can be used to study linear algebraic groups and their homogeneous spaces (e.g.~\cite{MR1230289,MR1230290,MR1381987,MR2286635,MR2488484}) as well as partial equivariant compactifications thereof (\cite{MR2999313,MR3117310} and the incomplete~\cite{arXiv:1006.3345}); the asymptotic formulas in the latter cases are interpreted in a way that is similar to the formula~\eqref{eq:Manin-Peyre} proposed by Manin and Peyre, building on the framework set out in~\cite{MR2740045}. 

Universal torsors were defined by Colliot-Thélène and Sansuc~\cite{MR899402} and were first used by Salberger~\cite{MR1679841} to count rational points on toric varieties. This method can be adapted to count integral points~\cite{10.1093/imrn/rnac048}, and in the present paper, we use it to count
integral points of bounded height on the toric variety defined as follows. Let $X_0 = \PP^1\times \PP^1 \times \PP^1$ be a product of projective lines with coordinate pairs $\pc{a_0:a_1}$, $\pc{b_0:b_1}$, and $\pc{c_0:c_1}$, and consider the two lines $l_1 = V(a_1,b_1)$ and $l_2=V(a_1,c_1)$. Let $\pi \colon X\to X_0$ be the toric variety obtained by blowing up $\PP^1\times \PP^1 \times \PP^1$ in $l_1$ and then blowing up the resulting variety in the strict transform of $l_2$. Denote by $T_0 = X_0 \setminus V(a_0a_1b_0b_1c_0c_1)$ the open torus in $X_0$ and by $T=\pi^{-1}(T_0)$ the open torus in $X$. Denote by $E_1$ and $E_2$ the two exceptional divisors above the lines $l_1$ and $l_2$ and by $M$ the preimage of the plane $V(a_0)$ parallel to them. Let
\[
  D=E_1+E_2+M, \quad U=X \setminus D,\quad\text{and}\quad  U_0 = X_0 \setminus (V(a_0)\cup l_1\cup l_2),
\]
the latter subvariety being isomorphic to $U$.
Consider the integral model
\[
  \mU = \PP^1_\ZZ \times \PP^1_\ZZ \times \PP^1_\ZZ \ \setminus \ 
  \overline{V(a_0)\cup l_1\cup l_2}
\]
of $U_0 \cong U$. An integral point on $\PP^1\times \PP^1\times \PP^1$ can be represented by three pairs $(a_0,a_1)$, $(b_0,b_1)$, $(c_0,c_1)$ of coprime integers, uniquely up to three choices of sign. Such a point lies in the complement of $V(a_0)$ if $a_0$ does not vanish modulo any prime, that is, if $a_0$ is a unit. Similarly, the point lies in the complement of $l_1$ if $a_1$ and $b_1$ do not simultaneously vanish modulo any prime, that is, if $\gcd(a_1,b_1)=1$, and analogously for $l_2$. The set of integral points on $\mU$ is thus
\begin{equation}\label{eq:integral-points-intro}
  \mU(\ZZ) = \left\{ \left(\pc{a_0:a_1},\, \pc{b_0:b_1},\, \pc{c_0:c_1}\right)\in X_0(\QQ) \ \relmiddle| 
  \substack{
    a_0,\dots,c_1\in \ZZ, \quad a_0\in\{\pm 1\}, \\
    \gcd(a_1,b_1)=\gcd(a_1,c_1)=\gcd(b_0,b_1)=\gcd(c_0,c_1)=1
  }
  \right\}.
\end{equation}
A point $P\in \mU(\ZZ)$ represented as in~\eqref{eq:integral-points-intro} lies in $T_0(\QQ)$ (resp.~in $T$ when interpreted as a point on $X$) if and only if all coordinates are nonzero, and for any such point, we set
\[
  H(P) = \abs{a_1}\max\{\abs{b_0},\abs{b_1}\}^2\max\{\abs{c_0}\abs{c_1}\}^2,
\]
which will turn out to be a log anticanonical height function on the pair $(X,D)$ (Lemma~\ref{lem:4-to-1-and-height}).

\begin{theorem}\label{thm:intro-count}
The number
  \begin{equation*}
    N(B)=\{P\in\mU(\ZZ)\cap T(\QQ)\mid H(P)\le B\}
  \end{equation*}
  of integral points of bounded height satisfies
  \begin{equation}\label{eq:result-formula}
    N(B)=c B (\log B)^2 + O(B\log B(\log\log B)^3)\text{,}
  \end{equation}
  where
  \begin{equation*}
    c=4 \prod_p \left(\left(1-\frac{1}{p}\right)^2\left(1+\frac{2}{p}-\frac{1}{p^2}-\frac{1}{p^3}\right)\right) \text{.}
  \end{equation*}
\end{theorem}

Thanks to the machinery in~\cite{MR2520770}, the proof is very straightforward (Sections~\ref{ssec:toric-torsor} and~\ref{ssec:toric-counting}). The main interest of the theorem lies in the shape of the asymptotic formula:
it
contradicts part of the unpublished preprint~\cite{arXiv:1006.3345} by
Chambert-Loir and Tschinkel, exemplifying a gap in a proof of which they
were already aware and due to which they no longer believed in the
correctness of their result (see Remark~\ref{rmk:gap} for more details on this issue).
Conceptually, the deviation from the asymptotic formula in op.~cit.\ can be explained
by an obstruction to the existence of integral points in a region that
was expected to dominate the number of integral points. 
More precisely, the maximal number of components of $D$ that have a common real point is part of the exponent of $\log B$ because integral points in arbitrarily small real neighborhoods of such intersections normally constitute 100\% of the total asymptotic number.

\begin{figure}
  \begin{center}
    \includegraphics{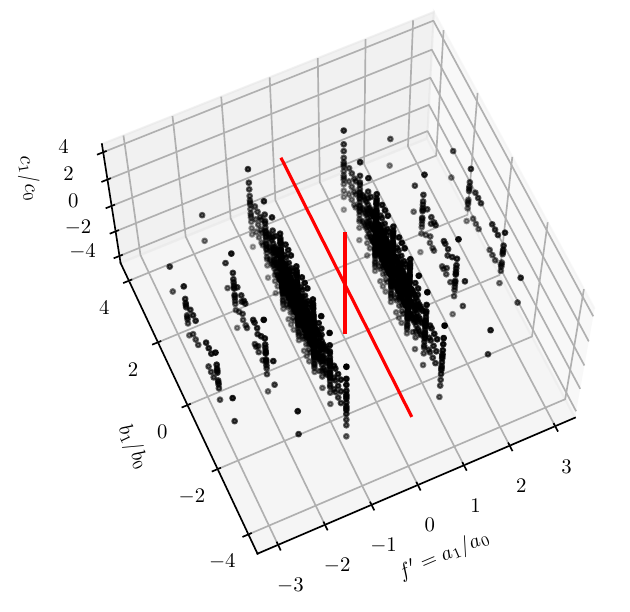}
    \caption[Integral points of height at most $9$ in $\mU(\ZZ)\cap T(\QQ)$, viewed as a subset of $\PP^1 \times \PP^1 \times \PP^1$]{Integral points of height at most $9$ in $\mU(\ZZ)\cap T(\QQ)$, viewed as a subset of $\PP^1 \times \PP^1 \times \PP^1$. The two lines $l_1$ and $l_2$ blown up are in the plane $a_1=0$. 
    By~\cite{arXiv:1006.3345}, one expects arbitrarily small neighborhoods of the intersection of the two red lines to dominate the counting function---but in fact, any sufficiently small such neighborhood contains no points counted by $N$ at all:
    as $a_1/a_0$ is an integer for all integral points, all integral points lie on ``sheets'', and all these sheets have distance $\ge 1$ from the intersection point, which corresponds to the unique \emph{maximal dimensional face of the Clemens complex}.
    (The plane $a_1/a_0=0$ defined by both lines contains integral points on $\mU$, which are not shown as they are not in $T(\QQ)$; in fact, it contains infinitely many points, all of height $1$, hence has to be discarded as an accumulating subvariety to achieve a well-defined counting function.)}\label{fig:toric-points}
  \end{center}
\end{figure}

On the toric variety $X\setminus D$, this does not hold: The function $f = a_1/a_0$ (in fact, a character of $T_0$) is regular on $U_0$ and spreads out to $\mU$; hence, it is an integer on every integral point $P\in\mU(\ZZ)$. If $P\in T_0(\QQ)$, then $f(P) \ne 0$. So $\abs{f(P)}\ge 1$ for $P\in \mU(\ZZ) \cap T_0(\QQ)$, that is, for every point $P$ counted by $N$ (Figure~\ref{fig:toric-points}).
The set $V=\pi^{-1}\{\abs{f} < 1\}$ is a neighborhood of $E_1(\RR) \cap E_2(\RR)$, and in fact even of $E_1(\RR)\cup E_2(\RR)$, in the analytic topology, as $f$ vanishes on these sets---but $V$ does not meet $\mU(\ZZ) \cap T(\QQ)$, whence cannot contribute to the counting function $N$. This leaves only the irreducible component $M$ of $D$ to contribute to the number $N(B)$ of points of bounded height. We interpret the asymptotic formula in a way that only takes this component into account (Theorem~\ref{thm:interpreted-count}), explaining the smaller power of $\log B$. Section~\ref{ssec:toric-interpretation} deals with the details of this interpretation and the comparison with loc.~cit.

\medskip

This phenomenon is an instance of a more general obstruction and related to the construction of a factor in the leading constant. 
After fixing notation and recalling the relevant definitions on Clemens complexes and the kind of Tamagawa measures appearing in the context of integral points (Section~\ref{ssec:metrics-heights-tamagawa}), we describe factors $\alpha_\uA$ associated with each \emph{maximal face} $\uA$ of the Clemens complex analogous to Peyre's constant $\alpha$, slightly generalizing a construction in~\cite{arXiv:1006.3345} to nontoric varieties (Section~\ref{ssec:divisor-group}, in particular Definition~\ref{def:effective}).

In~\eqref{eq:def-Delta_uA-U_uA}, we associate an open subvariety $U_\uA$ with every \emph{maximal face $\uA$ of the Clemens complex}. As soon as this open subvariety admits nonconstant regular functions, the set of integral points in a certain region of the variety associated with $\uA$ fails to be Zariski dense, and we say that there is an \emph{obstruction to the Zariski density of integral points near $\uA$};
in such a case, the maximal face $\uA$ cannot contribute to the counting function~\eqref{eq:integral-count-general}.
This obstruction is described and studied in detail in Section~\ref{ssec:obstruction}.
It turns out that it is closely related to an obstruction described by Jahnel and Schindler: under some assumptions on $X$ and $D$, including that they be defined over a field $K$ with only one archimedean place $\infty$, an obstruction at \emph{every} face of the archimedean Clemens complex implies that the open subvariety $X\setminus D$ is \emph{weakly obstructed at~$\infty$} in the sense of~\cite[Def.~2.2~(ii--iii)]{MR3736498} (Lemma~\ref{lem:face-obs-implies-obs-at-infinity}); in particular, its set of integral points is not Zariski dense. This allows a generalization of~\cite[Thm.~2.6]{MR3736498} to arbitrary number fields: the set of integral points on $U$ is not Zariski dense if all maximal faces of the Clemens complex are obstructed, without restrictions on $K$ (Corollary~\ref{cor:density-obstruction}), or, more generally, if a condition similar to being obstructed at infinity in the sense of op.~cit., simultaneously involving all infinite places, holds (Theorem~\ref{thm:var-density-obstruction}).

The relation between this obstruction and the construction of $\alpha_\uA$ is explored in Section~\ref{ssec:relation-obstruction-constant}: whenever this constant vanishes or some pathologies appear in its construction, the Zariski density of the corresponding set of integral points is obstructed (Theorem~\ref{thm:obstruction-construction}), providing a geometric reason for the face $\uA$ to be discarded. Finally, we briefly sketch how to take this obstruction into account when interpreting asymptotic formulas for the number of integral points of bounded height (Section~\ref{ssec:asymptotic-formulas}), Theorem~\ref{thm:interpreted-count} providing an example for this kind of interpretation.

\addtocontents{toc}{\SkipTocEntry}
\subsection*{Acknowledgements}

Part of this work was conducted as a guest at the
Institut de Mathématiques de Jussieu--Paris Rive Gauche invited by
Antoine Chambert-Loir and funded by DAAD\@.
During this time, I had interesting and fruitful discussions on
the interpretation of the result for the toric variety discussed
in Section~\ref{sec:toric} with Antoine Chambert-Loir.
I wish to thank him for these opportunities and for his useful remarks on
earlier versions of this article. Moreover, I want to thank Andrew O'Desky and the anonymous referee for several helpful remarks. This work was partly funded by
FWF grant P~32428-N35.

\section{Geometric framework}\label{sec:expectations}

Throughout this section, let $K$ be a number field, $\fo_K$ its ring of integers, $\Kbar$ an algebraic closure, $K_v$ the completion at a place $v$, and $k_v$ the residue field at a finite place $v$.
Equip the completions with the absolute values $\abs{\cdot}_v$
normalized such that
\[
  \abs{x}_v=\abs{N_{K_v/\QQ_w}(x)}_w
\]
at a place $v$ lying above a place $w$ of $\QQ$, such that $\abs{p}_p=1/p$ on $\QQ_p$, and with the usual absolute value on $\RR$. Moreover, equip each of the local fields with a Haar measure $\mu_v$ satisfying $\mu_v(\fo_{K_v})=1$ at finite places,
the usual Lebesgue measure $\diff\mu_v=\diff x$ at real places,
and $\diff\mu_v=i \diff z \diff \overline{z} = 2 \diff x \diff y$ at complex places.

We consider pairs $(X,D)$ as follows. Throughout, let $X$ be a smooth, projective, geometrically integral $K$-variety, and let $D$ be a reduced, effective divisor with strict normal crossings. Let $U=X\setminus D$, and let $\mU$ be an integral model, that is, a flat and separated $\fo_K$-scheme of finite type together with an isomorphism between its generic fiber $\mU \times_{\fo_K} K$ and $U$.
Throughout, similarly to~\cite[Déf~3.1, Hyp.~3.3]{MR2019019}, we assume that
\begin{enumerate}[label = (\arabic*)]
  \item $H^1(X,\cO_X)=H^2(X,\cO_X)=0$,
  \item the geometric Picard group $\Pic(X_\Kbar)$ is torsion free,
  \item there is a finite number of effective divisors $D_1,\dots, D_n$ that generate the pseudoeffective cone
  $\overline{\Eff}_X=\overline{\{\sum a_i D_i \mid a_i \in\RR_{\ge 0}\}} \subset {\Pic(X)}_\RR$ (which we shall also simply call the \emph{effective cone} for that reason), and
  \item the log anticanonical bundle ${\omega_X(D)}^\vee$ is big, that is, it is in the interior of the pseudoeffective cone.
\end{enumerate}
In particular, the anticanonical bundle $\omega_X^\vee$ is also big, and $X$ is \emph{almost Fano} in the sense of~\cite[Déf~3.1]{MR2019019}.

For simplicity, we will assume some form of \emph{splitness} of the pair $(X,D)$: that the canonical homomorphism $\Pic(X)\to\Pic(X_\Kbar)$ is an isomorphism and that all irreducible components of $D_\Kbar$ are defined over $K$. This assumption on $D$ is weaker than the pair $(X,D)$ being split in the sense of~\cite{MR3732687}.

To fix further notation, for an open subvariety $V\subset X$, we let
\[
  E(V)=\cO_X(V)^\times/K^\times
\]
be the finitely generated abelian group of invertible regular functions on $V$ up to constants.

\subsection{Metrics, heights, Tamagawa measures, and Clemens complexes}\label{ssec:metrics-heights-tamagawa}

To make this section self-contained, we begin by briefly recalling several definitions needed for the geometric interpretation of asymptotic formulas, as found for example in~\cite{MR2019019,MR2740045}.

\subsubsection{Adelic metrics}
Several of the invariants appearing in the interpretation of asymptotics formulas often depend on a choice of an \emph{adelic metric} on one or several line bundles $\cL$, that is, a family of norm functions
$\norm{\cdot}_v\colon \cL(x_v) \to \RR_{\geq 0}$
on the fibers $\cL(x_v) = x_v^* \cL$ above every local point $x_v\in X(K_v)$ that varies continuously with the points $x_v$ and is \emph{induced by a model} at almost all places~\cite[Ex.~2.2, Déf.~2.3]{MR2019019}.

  In an application of the torsor method, we shall make use of the following standard construction of an adelic metric on a base point free bundle $\cL$:
  a set $(s_0,\dots,s_n)$ of sections of $\cL$ without a common base point defines a morphism
  $f\colon X \to \PP^n$, $x\mapsto (s_0(x):\dots:s_n(x))$, and we set
  \[
    \norm{s(x)}_v =
    \min \left\{
      \abs{\frac{s(x)}{s_0(x)}}_v,
      \dots,
      \abs{\frac{s(x)}{s_n(x)}}_v
    \right\} 
  \]
  for a $K_v$-point $x=(x_0:\dots:x_n)$ and a local section $s$ of $\cL$. $\cO_{\PP^n}(1)$.
  Furthermore, adelic metrics $\norm{\cdot}_{\cL_1}, \norm{\cdot}_{\cL_2}$ on
  two line bundles $\cL_1$, $\cL_2$ induce an adelic metric on their quotient 
  $\cL_1 \otimes \cL_2^\vee$ by setting
  \[
    \norm{g(x)}_{v}
    = \frac{\norm{(g\otimes f_2)(x)}_{\cL_1,v}}{\norm{f_2(x)}_{\cL_2,v}}
  \]
  for a local section $g$ of $\cL_1 \otimes \cL_2^\vee$, which is independent of 
  a choice of a nonvanishing local section $f_2$ of $\cL_2$.

\subsubsection{Heights}
Such an adelic metric $\norm{\cdot}$  induces a height function
\[
  H_{\cL,\norm{\cdot}}\colon X(K)\to \RR_{\geq 0},
  \quad x\mapsto \prod_v \norm{s(x)}^{-1}_v\text{,}
\]
where $s$ is an arbitrary section that does not vanish in $x$~\cite[Déf~2.3]{MR2019019}.
If $\cL$ is ample, the number $\#\{x\in X(K)\mid H_\cL(x)\leq B\}$ of rational points of bounded height is finite for any $B\ge 0$. This still holds outside a closed subvariety if $\cL$ is big.

\begin{example}
In order to construct a height associated with an arbitrary (not necessarily base point free) line bundle $\cL$, something that we shall need in the second section, we can write the bundle as a quotient $\cL= \cA\otimes \cB^{-1}$ of base point free bundles.
Choosing sets $a_0,\dots,a_r$ and $b_0,\dots,b_s$ of global sections of $A$ and $B$ without a common base point induces morphisms
$f_1\colon X\to \PP^r$ and $f_2\colon X\to \PP^s$, respectively, as well as metrics on the two bundles.
Then
\[
  H_\cL=\prod_v \frac{\max\{\abs{a_i}_v\mid i=0,\dots,r\}}{\max\{\abs{b_j}_v\mid j=0,\dots,s\}} = \frac{H_{\PP^r}\circ f_1}{H_{\PP^s}\circ f_2},
\]
where $H_{\PP^r} = \prod_v \max\{\abs{x_0}_v,\dots,\abs{x_r}_v\}$ is the standard height on $\PP^r$, using that $\prod_v\abs{s(x)}_v=1$ for any section $s$ not vanishing in the respective image of $x$.

If $\cL$ is not base point free, it does not suffice to take global sections of $\cL$ and consider the maximum of their absolute values. There is, however, an inequality: If $\cL$ has global sections, take a basis $s_0,\dots,s_n$ of them. After completing $\{s_i b_j\}_{i,j}$ to a basis of the global sections of $\cA$, we get
\[
  H(x)\ge \prod_v \max_{i=0,\dots,n}\{\abs{s_i}_v\}
\]
for the height function $H$ induced by this choice of basis. If $x$ is not contained in the base locus, the right-hand side is $H_{\PP^n}(g(x))$ for the rational map $g\colon X\dashrightarrow \PP^n$ associated with $\cL$.
From this, we can recover the above fact: Assume that $\cL$ is big. 
After passing to an appropriate multiple (and taking the corresponding root of all resulting metrics and height functions), we can write $\cL=\cA\otimes \cE$, where $\cA$ is very ample and $\cE=\cO_X(E)$ is effective.
Endow $\cA$ and $\cE$ with metrics according to the above constructions and $\cL$ with the product metric. Let $f\colon X \hookrightarrow \PP^n$ be the resulting closed immersion corresponding to $\cA$ and $g\colon X\dashrightarrow \PP^m$ be the resulting rational map corresponding to~$\cE$. If $x\in (X\setminus E)(\QQ)$, then $g(x)$ is defined, and in particular $H_\cE(x) \ge H_{\PP^m}(g(x))\ge 1$; if moreover $H_\cL(x)\le B$, then
\[
  H_{\PP^n}(f(x)) = \frac{H_\cL(x)}{H_\cE (x)} \le B,
\]
making the number of rational points on $X\setminus E$ of height at most $B$ finite.
\end{example}

\subsubsection{Tamagawa measures}
A Tamagawa measure, appearing in asymptotic formulas for the number of rational points of bounded height, is a Borel measure $\tau_{X,v}$ on $X(K_v)$ for a place $v$, induced by an adelic metric on the canonical bundle $\omega_X$.
In local coordinates $x_1,\dots,x_n$, it is defined as
\[
  \diff \tau_{X,v}
  = \frac{\diff x_1 \cdots \diff x_n}{\norm{\diff x_1 \wedge \dots \wedge \diff x_n}_v},
\]
with respect to the fixed Haar measure on $K_v^n$~\cite[Not.~4.3]{MR2019019}. This is independent of the choice of local coordinates by the change of variables formula.

When counting integral points, variants of this measure induced by a metric on the log anticanonical bundle $\omega_X(D)$ take its place:
\begin{equation}\label{eq:tamagawa-finite}
  \diff \tau_{(X,D),v} = \frac{\diff x_1 \cdots \diff x_n}{\norm{1_D \otimes \diff x_1 \wedge \dots \wedge \diff x_n}_{v}}
\end{equation}
and its restriction to $U(K_v)$~\cite[\S~2.1.9]{MR2740045}. Here, $1_D$ denotes the canonical section of $\cO_X(D)$, corresponding to $1$ under the canonical embedding $\cO_X(D)\to \cK_X$. 
For a finite place $v<\infty$, the set $\mU(\fo_v)$ is compact, so the norm $\norm{1_D}^{-1}_v$ is bounded on it, and its volume is finite.
Moreover, these volumes verify
\[
  \tau_{(X,D),v}(\mU(\fo_v)) = \frac{\#\mU(k_v)}{(\#k_v)^{\dim X}}
\]
at almost all places (\cite[\S~2.4.1 with \S~2.4.3]{MR2740045} going back to Weil~\cite{MR670072}).
These measures are multiplied with convergence factors associated with the Galois representations $\Pic(U_{\Kbar})$ and $E(U_\Kbar)$~\cite[Def.~2.2]{MR2740045}. Our splitness assumption implies that the Galois action on both modules is trivial, and we get the powers
\[
  \left(1-\frac{1}{\#k_v}\right)^{\rk\Pic(U)-\rk E(U)}
\]
of the local factors at $s=1$ of the Dedekind zeta function $\zeta_K$ of $K$. These make the product
\[
  \prod_{v < \infty} \left(1-\frac{1}{\#k_v}\right)^{\rk\Pic(U)-\rk E(U)} \tau_{(X,D),v}(\mU(\fo_v))
\]
absolutely convergent~\cite[Thm.~2.5]{MR2740045}. Finally, this product is multiplied with the principal value of the corresponding $L$-function~\cite[Def.~2.8]{MR2740045}, in this case $\rho_K^{\rk\Pic U - \rk E(U)}$, where
\[
  \rho_K=\frac{2^r (2\pi)^s \Reg_K h_k}{\# \mu_K \sqrt{\abs{d_K}}}
\]
is the principal value of the Dedekind zeta function, with the numbers $r$ and $s$ of real and complex places, the regulator $\Reg_K$, the class number $h_K$, the group $\mu_K$ of roots of unity, and the discriminant $d_K$ of $K$.

\subsubsection{Clemens complexes}\label{sssec:Clemens-complexes}

Often, there are simple topological reasons for which integral points cannot evenly distribute along $U(K_v)$ for archimedean places $v$. For instance, if $K=\QQ$ and $U$ is affine, its integral points are lattice points, which form a discrete subset of $U(\RR)$; however, regarded as a subset of the compact set $X(\RR)$, the set of integral points must have accumulation points as soon as it is infinite, which consequently lie on the boundary $D(\RR)$.
Phenomena like this can be observed in many cases: integral points tend to accumulate near the boundary, with ``more'' points lying near intersections of several components of the boundary divisor. For this reason, combinatorial data on the boundary, encoded in Clemens complexes~\cite[\S~3.1]{MR2740045}, appears in asymptotic formulas for the number of integral points of bounded height.

The \emph{geometric Clemens complex} $\cC_\Kbar(D)$ is a partially ordered set defined as follows: Let $\cA$ be an index set for the set of irreducible components of $D$ (which are the same as the irreducible components of $D_\Kbar$ by our assumptions); denote by $D_\alpha$ the irreducible component of $D$ corresponding to $\alpha\in \cA$, and, for any $A\subset\cA$, by $Z_A$ the intersection $\bigcap_{\alpha\in A}D_\alpha$.
Then the geometric Clemens complex consists of all pairs $(A,Z)$, such that $A$ is a subset of $\cA$, and $Z$ is an irreducible component of $(Z_A)_\Kbar$.
Its ordering is given by $(A,Z)\preceq(A',Z')$ if $A\subset A'$ and $Z\supset Z'$. The dimension of a \emph{face} $(A,Z)\in \cC_\Kbar(D)$ is the longest length $n$ of a chain
\[
  (\emptyset, X) \prec (A_0,Z_0) \prec ...\prec (A_n,Z_n) = (A,Z)
\]
of strict inclusions.
In other words, we add a vertex (of dimension $0$) for every component of $D$; if the intersection of a set of components is nonempty, we glue one simplex to the corresponding set of vertices for every geometric component of the intersection. In the following, we will often suppress $Z$ from the notation.

For an archimedean place $v$, we will also be interested in the \emph{$K_v$-analytic Clemens complex} $\Can_{v}(D)$. It is the subset of $\cC_\Kbar(D)$ consisting of all pairs $(A,Z)$ such that $Z$ is defined over $K_v$ and has a $K_v$-rational point. (Note that this depends on $v$ and not just on the isomorphism class of $K_v$.) By the splitness assumption in the beginning of this section, we have the following:

\begin{lemma}
  If a face $(A,Z)$ of the geometric Clemens complex is part of the $K_v$-analytic Clemens complex $\Can_{v}(D)$, then so are all of its subfaces $(A^\prime,Z^\prime)$.
\end{lemma}
\begin{proof}
  Such a subface is given by data $A^\prime=\{D_1,\dots,D_r\}\subset A$ and  an irreducible component $Z^\prime \subset Z_{A^\prime}$ with $Z^\prime\supset Z$.
  Since $Z(K_v)\ne\emptyset$, it contains a $K_v$-point $P$, that is, a point $P$ invariant under the action of the Galois group of $K_v$; since $Z^\prime\supset Z$, the point $P$ also lies on $Z^\prime$.
  For contradiction, assume now that $Z^\prime$ is not defined over $K_v$. 
  Since the $D_i$ are all defined over $K$, they are invariant under the Galois action, and hence so is $Z_{A'}$; it follows that all conjugates $^\sigma\!Z^\prime$ of $Z^\prime$ are irreducible components of $Z_{A^\prime}$ as well. As $P$ is contained in the intersection of all conjugates (and there is more than one by the assumption), $Z_{A^\prime}$ is singular in $P$, and $D$ cannot have strict normal crossings.
\end{proof}

Observe that, since $Z_A$ is smooth for every $A\subset \cA$, the set $Z_A(K_v)$ is a smooth $K_v$-manifold.
We are interested in maximal faces of the analytic Clemens complex with respect to the ordering (i.e., \emph{facets}), corresponding to minimal strata of the boundary; such faces need not be maximal-dimensional, that is, maximal with respect to their number of vertices. Speaking geometrically, maximal faces are faces $(A,Z)$ such that $Z(K_v)$ intersects no other divisor component $D_\alpha(K_v)$, $\alpha\not\in A$.  We denote the set of maximal faces by $\Canmax_v(D)$; if the $K_v$-analytic Clemens complex is empty at a place $v$---that is, if $D(K_v)=\emptyset$---then the empty set is its unique maximal face.

\subsubsection{Archimedean analytic Clemens complexes}

Finally, we shall often be interested in all archi\-medean places at the same time. To this end, we define the \emph{archimedean analytic Clemens complex} 
to be the product (or \emph{join})
\[
  \Can_\infty(D) = \prod_{v\mid\infty} \Can_v(D)
\]
of all $K_v$-analytic Clemens complexes, that is, the product set with the induced partial order. In particular, a face $\uA\in\Can_\infty(D)$ is a tuple $\uA=(A_v)_{v\mid\infty}$ of faces $A_v\in\Can_v(D)$ and has dimension $\sum_v \#A_v -1$. 
Associated with such a face is the set
\begin{equation}\label{eq:Z_uA}
  Z_\uA = \prod_{v\mid\infty} Z_{A_v}(K_v);
\end{equation}
it is a closed subset of $X(K_\RR) = \prod_{v\mid\infty} X(K_v)$.

Finally, note that a face $\uA$ is maximal if and only if all the $A_v$ are; again, we denote by $\Canmax_\infty(D)$ the set of maximal faces.

\begin{remark}
  There are conflicting conventions on whether to include the empty face $(\emptyset, X)$ in such complexes. 
  Here, we include it, as a face of dimension $-1$. For instance, this convention is necessary to neatly deal with the case $D=0$, in which studying the empty face of the Clemens complex recovers Manin's conjecture on rational points; moreover, it facilitates the statement of some theorems that follow, in which the empty face is an interesting special case.
\end{remark}

\subsubsection{The measure associated with a maximal face}\label{sssec:residue-measure}

Let $v$ be an archimedean place, and let $A\in\Canmax_v(D)$ be a maximal face of the $K_v$-analytic Clemens complex, that is, a maximal subset of the irreducible components whose intersection $Z_A$ has a $K_v$-rational point.
Let
\[
  D_A = \sum_{\alpha \in A} D_\alpha \qquad \text{and} \qquad \Delta_A=D-D_A
\]
be the sum of divisors corresponding to $A$ and their ``complement'', respectively. We are interested in a measure $\tau_{Z_A}$ on $Z_A(K_v)$ defined as follows~\cite[\S\,2.1.12]{MR2740045}: A metric on $\omega_X(D_A)$ defines a metric on $\omega_{Z_A}$ and thus a Tamagawa measure $\tau$ on $Z_A(K_v)$ by repeated use of the adjunction isomorphism (since $D$ is assumed to have strict normal crossings). We are interested in the modified measure
$\norm{1_{\Delta_A}}^{-1}_{\cO(\Delta_A),v}\tau$.
This measure only depends on the metric on the log canonical bundle $\omega_X(D)$: this metric induces one on $\omega_{Z_A} \otimes \cO_X(\Delta_A)|_{Z_A}$ via the adjunction isomorphism, and the above measure is equal to
\[
  \norm{1_{\Delta_A}\big|_{Z_A} \otimes \diff x_1 \wedge \dots \wedge \diff x_s}^{-1}_{\omega_{Z_A} \otimes \cO_X(\Delta_A)|_{Z_A},v} \diff x_1 \cdots \diff x_s,
\]
for local coordinates $x_1,\dots, x_s$ on $Z_A(K_v)$. Note that the maximality of $A$ guarantees that $\norm{1_{\Delta_A}}_v^{-1}$ does not have a pole on $Z_A(K_v)$, and is thus bounded on the compact set $Z_A(K_v)$.
These measures are further renormalized by a factor $c_{K_v}^{\# A}$, where $c_\RR=2$, resp. $c_\CC=2\pi$, is the volume of the unit ball in the archimedean local field with respect to the Haar measure we are using.
This results in the \emph{residue measure}
\[
  \tau_{Z_A} = c_{K_v}^{\# A} \norm{1_{\Delta_A}}^{-1}_{\cO(\Delta_A),v}\tau
\]
on $Z_A(\RR)$. See~\cite[\S\S\,3.3.1,~4.1]{MR2740045} for more details.

\begin{remark}
  Let $\uA \in \Canmax_\infty(D)$ be a maximal face of the archimedean analytic Clemens complex.
  The above constructions furnish a finite measure
  \[
    \tau_\uA = 
    \prod_{v < \infty} \left(1-\frac{1}{\#k_v}\right)^{\rk\Pic(U)-\rk E(U)} \tau_{U,v}
    \times \prod_{v \mid \infty} \tau_{Z_{A_v},v} 
  \]
  on the subset
  \[
    \prod _{v<\infty} \mU(\fo_v) \times \prod_{v\mid\infty}Z_{A_v}(K_v) \subset X(\bAA_K)
  \]
  of adelic points. Often, integral points of bounded height \emph{equidistribute} towards the sum of (suitably normalized) measures on a disjoint union of such sets: usually that over all \emph{maximal dimensional} faces, but Sections~\ref{ssec:obstruction} and~\ref{sec:toric} provide a reason for which this can fail. Chambert-Loir and Tschinkel provide an abstract theorem that allows such an equidistribution theorem to be deduced from sufficiently general counting theorems~\cite[Prop.~2.10]{MR2740045}, generalizing work of Peyre's~\cite[Prop.~5.4]{MR2019019} in the setting of rational points.
\end{remark}

\subsection{A divisor group and the constants \texorpdfstring{$\alpha_\uA$}{alpha A}}\label{ssec:divisor-group}

We associate some data with tuples of maximal faces analogous to groups defined by Chambert-Loir and Tschinkel for toric varieties~\cite[\S\,3.5]{arXiv:1006.3345}, using the full set of divisors instead of invariant ones:

\begin{definition}
  Let $\uA\in\Canmax_\infty(D)$.
  \begin{enumerate}
    \item An \emph{$\uA$-divisor} $L$ is a tuple $(L_U, (L_v)_{v\mid\infty})$, where $L\in \Div(U)$ is a divisor on $U$ and $L_v \in \Div(X)$ are divisors supported on $D_{A_v}$, that is, linear combinations of $D_\alpha$ with $\alpha\in A_v$. Denote by 
    \[
    \Div(U;\uA)\cong\Div(U) \oplus \bigoplus_{v\mid\infty} \ZZ^{A_v}
    \]
    the \emph{group of $\uA$-divisors}.

    \item Let  $\sdiv_\uA\colon \cK_X \to \Div(U;\uA)$ be the map associating with a rational function $f$ its corresponding \emph{principal $\uA$-divisor}
    \[
      \sdiv_\uA(f) = \left(\sdiv_U(f),\ \left(\sum_{\alpha\in A_v} \ord_{D_\alpha}(f) D_\alpha\right)_v \, \right).
    \]
    \item Finally, denote by
    \[
        \Pic(U;\uA)=\Div(U;\uA)/\im(\sdiv_\uA)
    \]
    the \emph{group of $\uA$-divisor classes}.
  \end{enumerate}
\end{definition}

\begin{notation}
  Since $\sdiv_\uA$ is compatible with the standard divisor function, the pullback homomorphism $i^*\colon \Pic(X) \to \Pic (U)$ along the inclusion $i\colon U\to X$ factors through $\Pic(U;\uA)$ as follows.
  The first homomorphism
  \begin{equation}
      \label{eq:pi_A} \pi_\uA\colon \Pic(X)\to\Pic(U;\uA)
  \end{equation}
  maps the class $[L]$ of an irreducible divisor $L\in \Div(X)$ to
  \[
    \pi_{\uA}([L]) = [(L\cap U,(\delta_{L\in A_v} L)_v)], 
  \]
  where $\delta_{L\in A_v}$ is $1$ if $L$ is a component of $D$ and belongs to the maximal face $A_v$ and $0$ otherwise. The second homomorphism
  \begin{equation}
      \label{eq:sigma_A} \sigma_\uA\colon\Pic(U;\uA)\to\Pic(U)
  \end{equation}
  maps the class of an $\uA$-divisor $(L_U, (L_v)_v)$ to the class
  \[
    \sigma_{\uA}([L_U, (L_v)_v]) = [L_U]
  \]
  of $L_U$ in $\Pic U$.
  If $K$ has only one archimedean place and $A\in \Canmax_\infty(D)$, these constructions simplify to the Picard group $\Pic(U;A)=\Pic(U_A)$ of $U_A$ and the pullback homomorphisms $\Pic(X)\to\Pic(U_A)$ and $\Pic(U_A)\to\Pic(U)$.
  Moreover, there is an isomorphism
  \begin{align*}
    \Pic(U;\uA) &\to \left\{(L_v)_v \in \bigoplus_{v\mid\infty} \Pic(U_{A_v})\relmiddle| L_v|_{U}\cong L_{v'}|_{U} \text{ for all $v,v'\mid\infty$}\right\} \\
    [L, (L_v)_v] &\mapsto ([L+L_v])_v
  \end{align*}
  that allows an equivalent interpretation of a $\uA$-divisor class as a tuple of divisor classes (or line bundles) on $U_{A_v}$ whose restrictions to $U$ coincide. In particular, $\sigma_\uA$ factors through a canonical homomorphism
  \[
    \sigma_{\uA,v} \colon \Pic(U;\uA) \to \Pic(U_{A_v})
  \]
  for each $v\mid\infty$.
\end{notation}

Set
\begin{equation}\label{eq:def-Delta_uA-U_uA}
  \Delta_{\uA}
  =\sum_{\substack{\alpha\not\in A_v\\\text{for all } v\mid\infty}}
  D_\alpha
  \qquad \text{and} \qquad
  U_\uA=X \setminus \Delta_{\uA},
\end{equation}
so that $\Delta_\uA \subset D$ is again the ``complement'' of $\uA$. With $\Delta_{A_v}$ as before for the maximal face $A_v$ at a place $v$ and $U_{A_v}=X\setminus \Delta_{A_v}$, there are inclusions
\[
    \Delta_\uA \subset \Delta_{A_v} \subset D
    \qquad \text{and} \qquad
    U\subset U_{A_v}\subset U_\uA \subset X.
\]

\begin{lemma}\label{lem:sequences}
  The sequences
  \begin{equation}\label{eq:localization-seq}
    \begin{tikzcd}[column sep=scriptsize]
      0  \arrow[r]
      &  E(U) \arrow[r, "\sdiv"]
      &  \CH^0(D) \arrow[r, "(i_D)_*"]
      &  \Pic(X) \arrow[r, "i_U^*"]
      &  \Pic(U) \arrow[r]
      &  0,
    \end{tikzcd}
  \end{equation}
  where $i_D\colon D\to X$ and $i_U\colon U\to X$ are the inclusions,
  and
  \begin{equation}\label{eq:sequence-pic-u-a}
    \begin{tikzcd}[cramped, column sep=scriptsize]
      0 \arrow[r]
      & E(U_\uA) \arrow[r, "\cdot|_U"]
      & E(U) \arrow[r, "o"]
      & \bigoplus_{v\mid\infty} \CH^0(D_{A_v}) \arrow[r, "p"]
      & \Pic(U;\uA) \arrow[r, "\sigma_{\uA}"]
      & \Pic(U) \arrow[r]
      & 0,
    \end{tikzcd}
  \end{equation}
  where $o=(\ord_{D_\alpha})_{\alpha,v}$ and $p$ is induced by the projection $\Div(U;\uA) \to \Pic(U;\uA)$,
  are exact.
\end{lemma}
\begin{proof}
  The part on the right of~\eqref{eq:localization-seq} is the localization sequence for Chow groups. Exactness on the left follows from the observation that a relation making a divisor supported on $\supp(D)$ linearly trivial has to come from a meromorphic section whose only zeroes and poles are on $\supp(D)$, that is, an invertible regular function on $U$. The only such functions mapping to $0$ in $\CH^0(D)$ are regular and invertible on $X$, hence invertible constants, and we get $E(U)$ on the left.

  Turning to~\eqref{eq:sequence-pic-u-a}, the kernel of $\sigma_\uA$ is generated by $\uA$-divisors supported outside $U$, that is, on the $D_{A_v}$.
  If such an $\uA$-divisor $L=(0,(L_v)_v)$ is linearly equivalent to $0$ in $\Pic(U;\uA)$, this equivalence has to be induced by a section which has corresponding zeroes and poles on $L$, but no zeroes and poles on $U$; again, we can exclude constants. Finally, the invertible regular functions on $U$ not inducing such a relation, thus mapped to $0$ by $o$, are those which do not have a zero or pole on any $\uA$, that is, those that are regular and invertible on $U_\uA$.
\end{proof}

In the context of asymptotic formulas, we will be interested in the two numbers
\begin{align*}
  b_\uA  &= \rk\Pic(U)-\rk E(U) + \sum_{v\mid\infty}\# A_v \quad \text{and}\\
  b_\uA^\prime &= \rk\Pic(U;\uA)
\end{align*}
connected to the exponent of $\log B$ and a factor of the leading constant associated with $\Pic(U;\uA)$.

\begin{lemma}\label{lem:descriptions-b_A}
  For every $\uA\in\Canmax_\infty(D)$, we have
  \begin{align*}
    b_\uA        &= \rk\Pic X - \# \cA + \sum_{v\mid\infty}\# A_v \quad \text{and}\\
    b_\uA^\prime &= b_\uA + \rk E(U_\uA).
  \end{align*}
\end{lemma}
\begin{proof}
  These equalities follow directly from Lemma~\ref{lem:sequences}, on noting that 
  $\rk\CH^0(D)=\# \cA$.
\end{proof}

\begin{remark}
  An equality $b_\uA=b_\uA^\prime$ always holds for toric varieties~\cite[\S\,3.7.1 with the remark before Lem.~3.8.5]{arXiv:1006.3345} and partial equivariant compactifications of vector groups and semisimple groups (since their effective cones are simplicial and generated by invariant divisors, so there cannot be an element in $E(U_\uA)$ that would induce a relation between some of them). More generally, both numbers play a role in asymptotic formulas, and we shall see in Theorem~\ref{thm:obstruction-construction}~\eqref{enum:b_A-unequal} that they are equal whenever we can expect a tuple $\uA$ of maximal faces to contribute to an asymptotic formula.
\end{remark}

\begin{remark}
  As a consequence of the splitness assumption, it does not matter whether we work over $K$ or $\Kbar$: for a group $\Pic(U_\Kbar;\uA)$ similarly defined over $\Kbar$, there would be a canonical isomorphism $\Pic(U_\Kbar;\uA)\cong\Pic(U;\uA)$. Indeed, consider the exact sequences in Lemma~\ref{lem:sequences} over both $K$ and $\Kbar$ together with the obvious homomorphisms from the former to the latter.
  The splitness assumptions imply that the homomorphisms $\Pic(X)\to\Pic(X_\Kbar)$ and $\CH^0(D)\to \CH^0(D_\Kbar)$ are isomorphisms, so using the five lemma three times yields $\Pic(U_\Kbar;\uA)\cong\Pic(U;\uA)$.
\end{remark}

In~\cite{arXiv:1006.3345}, Chambert-Loir and Tschinkel define a group $\Pic(U;\uA)$ for (not necessarily split) toric varieties starting with torus-invariant divisors. Their construction coincides with the one above.

\begin{lemma}
  If $X$ is a split toric variety, $D$ is invariant under the action of the torus, and $\uA\in\Canmax_\infty(D)$, then 
  \begin{equation}\label{eq:toric-comparison}
    \Pic(U;\uA)
    \cong \left( \Pic^T(U) \oplus \bigoplus_{v\mid \infty} \ZZ^{A_v}\right)/\sdiv_\uA(M), 
  \end{equation}
  where $\Pic^T(U)$ is the group of torus invariant divisors and $M\subset \cK_X$ is the character group of the torus.
\end{lemma}
\begin{proof}
  The group on the right-hand side of~\eqref{eq:toric-comparison} fits into the sequence~\eqref{eq:sequence-pic-u-a} in place of $\Pic(U;\uA)$ on noting that $E(U), E_{\uA}(U)\subset M$ and that $\Pic(U)$ is generated by $\Pic^T(U)$.
  Then the five lemma implies that the inclusion $\Pic^T(U)\to \Div(U)$ induces the desired isomorphism.
\end{proof}

\subsubsection*{Effective cones and the $\alpha$-constant}

If $V$ is a real vector space and $\Lambda\subset V$ a lattice, 
we can equip the dual space $V^\vee$ with a Haar measure such that $\Lambda^\vee$ has covolume $1$.
Recall that if
$C\subset V$ a convex cone,
the \emph{characteristic function} $\charfun_{C}$ of $C$ is defined to be
\[
  \charfun_{C}\colon V\to \RR_{\ge 0}\cup\{\infty\} \qquad x\mapsto
  \int_{C^\vee} e^{-\langle x,t\rangle}\diff t.
\]
It is finite in the interior of $C$.

\begin{definition}\label{def:effective}
  Let $\uA\in \Canmax_\infty(D)$.
  \begin{enumerate}
    \item Let $V_\uA=\Pic(U;\uA)_\RR$, and define the \emph{effective cone} associated with $\uA$ to be the cone $\Eff_\uA\subset\Pic(U;\uA)_\RR$ generated by the images of effective divisors
    $\Div_{\ge 0}(U)_\RR\oplus\bigoplus_v\RR_{\ge 0}^{A_v}$. 
    \item\label{enum:alpha} Assume that $\Pic(U;\uA)$ is torsion free; it is thus a lattice in $V_\uA$. Define
    \[
      \alpha_{\uA}=\frac{1}{(b_\uA^\prime - 1)!}\charfun_{\Eff_\uA}(\pi_\uA (\omega_X(D)^\vee)),
    \]
    where $\pi_\uA\colon\Pic(X)\to\Pic(U;\uA)$ is as in~\eqref{eq:pi_A}.
  \end{enumerate}
\end{definition}

\begin{remark}
  \leavevmode

  \begin{enumerate}
    \item If $K$ has only one archimedean place and $A\in \Canmax_\infty(D)$, then $\Eff_A$ is simply the effective cone $\overline{\Eff}_{U_A}$ of $U_A$.
  
    \item In Theorem~\ref{thm:obstruction-construction}~\eqref{enum:torsion}, we shall see that the assumption in Definition~\ref{def:effective}~\eqref{enum:alpha} holds whenever we can expect $\uA$ to contribute to an asymptotic formula.
    
    \item Since the log anticanonical bundle $\omega_X(D)^\vee$ is big, its image is in the interior of $\Eff_\uA$, making $\alpha_\uA$ finite.
    
    This value is nonzero if and only if $\Eff_\uA$ is strictly convex. In Theorem~\ref{thm:obstruction-construction}~\eqref{enum:nonconvex}, we shall see that if this is not the case, then there is an obstruction to the Zariski density of integral points ``near $\uA$'', and the face $\uA$ should not contribute to an asymptotic formula.
    
    \item The constant $\alpha_\uA$ can alternatively be described as a volume: Equip the hyperplanes $H_a=\{t\in V_\uA^\vee\mid\langle\cL,t \rangle=a\}$ with measures $\nu_{H_a}$ normalized such that
    \[
      \int_{V_\uA^\vee} f \diff\nu  = \int_\RR  \left(\int_{H_a} f \diff \nu_{H_a}\right) \diff a
    \]
    for all functions $f$ on $V_\uA^\vee$ with compact support.
    Then (cf.~\cite[Ch.~2,~\S\,2]{MR0158414} and~\cite[Prop.~5.3]{MR1620682})
    \begin{align*}
      \alpha_\uA
      & = \nu_{H_1}\{t\in \Eff_\uA^\vee \mid
               \langle \omega_X(D)^\vee, t\rangle = 1 \}\\
      & = b_\uA^\prime\,
          \nu\{t\in \Eff_\uA^\vee \mid
               \langle \omega_X(D)^\vee, t\rangle \le  1 \}.
    \end{align*}
    If the cone $\Eff_\uA$ is \emph{smooth}, that is, generated by a $\ZZ$-basis $r_1,\dots,r_{b_\uA^\prime}$ of $\Pic(U;\uA)$, this can be simplified further: If $\pi(\omega(D)^\vee)$ has the representation $(a_1,\dots,a_{b_\uA^\prime})$ in this basis, we have
    \[
      \alpha_\uA = \frac{1}{(b_{\uA}^\prime-1)!} \,
      \prod_{1\le i \le b_\uA^\prime} \frac{1}{a_i}.
    \]
    More generally, if $\Eff_\uA$ is generated by a $\ZZ$-basis of a sublattice $\Lambda\subset \Pic(U;\uA)$, the same is true after dividing the right-hand side by the index $[\Pic(U;\uA):\Lambda]$.
  \end{enumerate}
\end{remark}

\subsection{An obstruction}\label{ssec:obstruction}

Let $\uA\in\Canmax_\infty(D)$ be a maximal face of the archimedean analytic Clemens complex, and consider the regular functions $\cO_X(U_\uA)$ on $U_\uA$.
If $\cO_X(U_\uA)\ne K$, that is, if there are nonconstant functions on $U_\uA$, then there are no integral points that are simultaneously near all $Z_{A_v}$, except possibly on a finite set of strict subvarieties (Corollary~\ref{cor:obstruction}). If any such subvariety were to contribute to the asymptotic behavior, we would have to exclude it as accumulating; hence, there cannot be a contribution of ``points near $\uA$'' to an asymptotic formula if $\cO_X(U_\uA)\ne K$. In this case, we will say that there is \emph{an obstruction to the Zariski density of integral points near $\uA$}.

The most general statement of this obstruction deals with not necessarily maximal faces $\uA$. In this setting, the existence of such a function prevents the density of points ``near $\uA$'', except possibly near strictly larger faces $\uB$: 

\begin{proposition}\label{prop:general-obstruction}
  Let $\uA\in\Can_\infty(D)$ be a (not necessarily maximal) face such that $\cO_X(U_\uA)\ne K$.
  For each $v\mid\infty$, let $B_{v,1},\dots,B_{v,r_v}\in \Can_v(D)$ be the faces strictly containing $A_v$.
  For each $v\mid\infty$ and $1\le i\le r_v$, let $W_{v,i}$ be an arbitrary analytic neighborhood of $Z_{B_{v,i}}(K_v)$. 

  Then there exists an analytic neighborhood $U_v$ of 
  \[
    Z_{A_v}(K_v) \setminus \bigcup_{i=1}^{r_v} W_{v,i}
  \]
  for each $v\mid\infty$ and a Zariski dense open subvariety $V\subset X$ such that
  \[
    \{ \mU(\fo_K) \cap V(K) \mid x \in U_v \text{ for all } v\mid\infty\} = \emptyset.
  \]
\end{proposition}
\begin{proof}
  Let $s$ be a nonconstant regular function on $U_\uA$. After multiplying with a suitable constant, we can assume it is regular on the integral model. Let $v$ be an infinite place. 
  As the poles of $s$ are contained in $\Delta_{A_v}$, its only poles on $Z_{A_v}(K_v)$ are contained in $Z_{A_v} \cap \Delta_{A_v} = \bigcup_i Z_{B_{v,i}}$.  Hence, $\abs{s}_v$ is continuous on the compact set $Z = Z_{A_v}\setminus\bigcup_i W_{B_{v,i}}$ and attains its maximum $M_v$.
  It follows that
  \[
    U_v=\{x\in X(K_v)\mid \abs{s(x)}_v < 2 M_v\},
  \]
  is a neighborhood of $Z$.
  Since $s(x)\in\fo_K$ for integral points $x\in\mU(\fo_K)$, it can only attain the finitely many integral values $\alpha$ in the box in $\prod_{v\mid\infty} K_v$ defined by the $M_v$. Every integral point lying in all the $U_v$ must thus lie on one of the finitely many subvarieties $\overline{V(s-\alpha)}\subset X$.
\end{proof}

\begin{corollary}\label{cor:obstruction}
  Let $\uA\in\Canmax_\infty(D)$ be a maximal face such that $\cO_X(U_\uA)\ne K$. Then there is a nonempty Zariski open subset $V\subset X$ and an analytic neighborhood $U_v$ of $Z_{A_v}$ in $X(K_v)$ for every archimedean place $v$ such that
  \[
    \{x\in \mU(\fo_K) \cap V(K) \mid x\in U_v \text{ for all } v\mid\infty\} = \emptyset.
  \]
\end{corollary}
\begin{proof}
  This is the special case of Proposition~\ref{prop:general-obstruction} where $\uA$ is maximal. Then for each $v\mid\infty$, there is no strictly larger $B_v\succ A_v$, and $U_v$ is simply a neighborhood of $Z_{A_v}(K_v)$.
\end{proof}

The obstruction can be rephrased using the notation~\eqref{eq:Z_uA}:

\begin{corollary}\label{cor:obstruction-adelic}
  Let $\uA\in\Can_\infty(D)$. For each strictly larger $\uB \succ \uA$, let $W_\uB$ be an open neighborhood of $Z_\uB$. Then there exists an open neighborhood $U$ of $Z_\uA \setminus \bigcup_{\uB\succ \uA} W_\uB$ and a nonempty Zariski open $V\subset X$ such that
  \[
    \mU(\fo_K) \cap V(K) \cap U = \emptyset,
  \]
  where the intersection is to be understood in $X(K_\RR)$.
  In particular, if $\uA$ is maximal, then
  \[
    \overline{\mU(\fo_K) \cap V(K)} \cap Z_{\uA} =\emptyset. 
  \]
\end{corollary}

\begin{remark}\label{rem:nonmaximal}
  If there is an obstruction to the Zariski density of integral points near a maximal face $\uA$ (in the sense of Corollary~\ref{cor:obstruction}), Proposition~\ref{prop:general-obstruction} implies that the density of points near $Z_{\uA^\prime}$ for subfaces $\uA^\prime$ is similarly obstructed---except possibly near a larger face $\uB \succ \uA^\prime$; indeed, any nonconstant regular section on $U_\uA$ is also regular on $U_{\uA^\prime}\subset U_\uA$.
  In such a case, we would expect that the number of such points is described by invariants attached to the maximal faces $\uB$ containing $\uA^\prime$. In other words, Corollary~\ref{cor:obstruction} provides no reason for which nonmaximal faces might have to be studied when counting integral points. 
\end{remark}

As a consequence, if all maximal faces $\uB$ containing a face $\uA$ are obstructed, there are no integral points simultaneously near all $Z_{A_v}(K_v)$:

\begin{corollary}\label{cor:subfaces}
  Let $\uA\in\Can_\infty(D)$ be a (not necessarily maximal) face of the archimedean analytic Clemens complex such that $\cO_X(U_\uB)\ne K$ for all maximal $\uB\succeq \uA$.
  Then there is a nonempty Zariski open subset $V\subset X$ and a neighborhood $U$ of $Z_\uA$ in $X(K_\RR)$ such that
  \[
    \{x\in \mU(\fo_K) \cap V(K) \mid x\in U_v \text{ for all } v\mid\infty\} = \emptyset.
  \]
\end{corollary}
\begin{proof}
  This follows using induction on the codimension $c$ of $\uA$.
  The case $c=0$ is Corollary~\ref{cor:obstruction}.
  For general $\uA$, if $\uB \succ \uA$, then the induction hypothesis implies the existence of a neighborhood $U_\uB$ of $Z_\uB$ and a nonempty Zariski open $V_\uB$ such that $\mU(\fo_K)\cap V_\uB\cap U_\uB = \emptyset$. Corollary~\ref{cor:obstruction-adelic} guarantees the existence of a neighborhood $U_\uA$ of $Z_\uA \setminus \bigcup_{\uB_\succ \uA} U_\uB$ and a nonempty Zariski open $V_\uA$ such that $\mU(\fo_K)\cap V_\uA\cap U_\uA = \emptyset$.
  The statement follows with $U = U_\uA \cup \bigcup_{\uB\succ\uA} U_\uB$ and $V = V_\uA \cap \bigcap_{\uB\succ\uA} V_\uB$.
\end{proof}

A notable special case is the empty face $(\emptyset,\dots,\emptyset)$, with corresponding stratum $Z_{\emptyset}=X$:

\begin{corollary}\label{cor:density-obstruction}
  Assume that $\cO_X(U_\uA)\ne K$ for all maximal faces $\uA \in \Canmax_\infty(D)$. Then $\mU(\fo_K)$ is not Zariski dense for any integral model $\mU$ of $U$.
\end{corollary}

This obstruction is very similar to the notion of a \emph{weak obstruction at infinity} developed by Jahnel and Schindler~\cite[Def.~2.2]{MR3736498}. For an archimedean place $v$, the complement $U$ of a very ample divisor $D$ is called \emph{weakly obstructed at $v$}
if there is a connected component $U^\prime$ of $U(K_v)$,
a constant $c>0$, an integer $d>0$, and a finite set of rational functions of the form $s_i=f_i/1_D^d$ with $f_i\in H^0(X,\cO_X(D)^{\otimes d})$ not multiples of $1_D^d$ (that is, nonconstant regular functions $s_i$ on $U$)
such that, for every point $x\in U'$, there is at least one $s_i$ with $\abs{s_i}_v < c$.

\begin{lemma}\label{lem:face-obs-implies-obs-at-infinity}
  Let $v$ be an archimedean place of $K$ and assume that $U(K_v)$ is connected. If $\cO_X(U_A)\ne K$ for all maximal faces $A$ of the $K_v$-analytic Clemens complex, then $U$ is weakly obstructed at $v$.
\end{lemma}
\begin{proof}
  Take a nontrivial $s_A\in \cO_X(U_A)$ for all maximal faces $A$. For every point $x$ on the boundary, at least one of the $s_A$ is regular in $x$. Moreover, all of them are regular on $U$, whence $\{\abs{s_i}<c\}_{i,c}$ covers the compact set $X(K_v)$, and there is a finite subcover. We can then take $c$ as the maximal constant used in this subcover.
\end{proof}

\begin{remark}
  Over fields with only one infinite place, integral points are not Zariski dense if $U$ is weakly obstructed at $\infty$ by~\cite[Thm.~2.6]{MR3736498}. In a more general setting, this does not need to be the case, even if $U$ is obstructed at every archimedean place (Example~\ref{ex:gm}).
  However, the above Corollary~\ref{cor:density-obstruction} and the following, more general Theorem~\ref{thm:var-density-obstruction} generalize the obstruction to arbitrary number fields.
\end{remark}

This obstruction always vanishes after a suitable base change:

\begin{lemma}\label{lem:obstruction-vanishes}
  There is a finite extension $L\supset K$ such that there is a maximal face $\uA=(A_w)_w$ of the archimedean analytic Clemens complex $\Can_\infty(D_L)$ with $\cO_{X_L}((U_L)_\uA) = L$.
\end{lemma}
\begin{proof}
  Let $A_1,\dots,A_n$ be the maximal faces of the geometric Clemens complex $\cC_\Kbar(D)$, and let $L\supset K$ be an extension with at least $n$ complex places $w_1,\dots,w_n$ (i.e., nonconjugate embeddings into $\CC$ whose image is not contained in $\RR$). Then $\Can_{L_{w_i}}(D_L)=\cC_\Kbar(D)$ for these places as all $Z_{A_i}(L_{w_i})$ are nonempty, and we can take the face $\uA=(A_w)_w$ with $A_{w_i}=A_i$ for these $n$ complex places and $A_w$ an arbitrary maximal face for all other places. Since every $D_i$ belongs to at least one maximal face of the geometric Clemens complex, we have $(U_L)_\uA=X_L$, hence $\cO_{X_L}((U_L)_\uA)=L$.
\end{proof}

\begin{figure}
  \begin{center}
    \includegraphics[scale = .8]{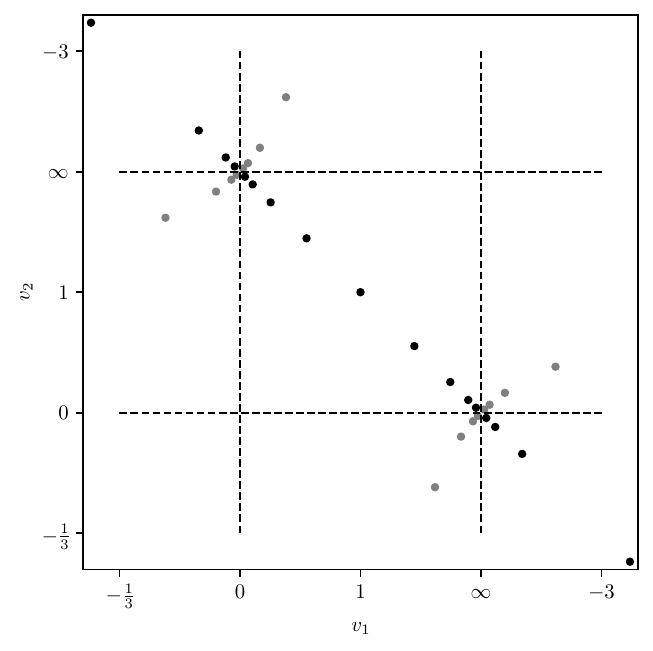}
    \caption{Integral points on $\multgrp$ over $K=\QQ(\sqrt{5})$---that is, units of the ring of integers---of small height. Those of norm $1$ are shown in black, those of norm $-1$ in grey. They are embedded into $\multgrp(\RR)\times \multgrp(\RR) = \RR^\times \times \RR^\times$ along its two places $v_1$ and $v_2$; a chart showing both $0$ and $\infty$ is used, and the complement of the multiplicative group is designated by dashed lines. The maximal faces $(0,0)$ and $(\infty,\infty)$ of $\Can_\infty(D)$ are obstructed, and no points are near them; the other two maximal faces $(0,\infty)$ and $(\infty, 0)$ are not, and integral points accumulate near them.
    }\label{fig:gm}
  \end{center}
\end{figure}

\begin{example}\label{ex:gm}
  Let $X=\PP^1 = \Proj K[t_0,t_1]$, $D=\{0,\infty\}$, and
  $U=\PP^1\setminus D = \multgrp$ with integral model $\mU = \mathbb{G}_{\mathrm{m},\fo_K}$. (Its log anticanonical bundle is trivial, violating the assumptions at the beginning of the section.) For every archimedean place $v$, the analytic Clemens complex consists of the two vertices $0$ and $\infty$, which are maximal faces.

  Assume for now that $K$ has only one infinite place; that is, $K$ is the field of rational numbers or imaginary quadratic. The two open subvarieties associated with the two maximal faces are $U_0 = \PP^1 \setminus \{\infty\}$
  and $U_\infty=\PP^1 \setminus \{0\}$. Considering $t=t_1/t_0 \in \cO_{\PP^1}(U_0)$
  and $t^{-1} \in \cO_{\PP^1}(U_\infty)$, both maximal faces are obstructed, and indeed, $\mathbb{G}_{\mathrm{m},\fo_K}(\fo_K) = \fo_K^\times$ is well known to be finite, whence not Zariski dense.

  If $K$ has more than one place, there are more maximal faces to consider.
  Choosing the maximal face $0$ at every place, we again have
  $U_{(0,\dots,0)} = \PP^1 \setminus \{\infty\}$, and $\uA=(0,\dots, 0)$ is obstructed by $t$. Analogously, $(\infty,\dots,\infty)$ is obstructed by $t^{-1}$.
  However, all remaining tuples $\uB$ of faces (e.g. $\uB=(0,\infty,\dots)$) satisfy $U_\uB=\PP^1$, whence $\cO_{\PP^1}(U_\uB) = K$, and those tuples are unobstructed. Indeed, in this case, the set $\mathbb{G}_{\mathrm{m},\fo_K}(\fo_K) = \fo_K^\times$
  is well known to be infinite, whence Zariski dense (and can be checked to be ``dense near'' all these $\uB$ as a consequence of Dirichlet's unit theorem). See Figure~\ref{fig:gm} for an example involving a real quadratic field.
\end{example}

Finally, Corollary~\ref{cor:density-obstruction} can be generalized and stated in a setting without assumptions on $U$, in a way and with a proof that is very similar to~\cite{MR3736498}:

\begin{theorem}\label{thm:var-density-obstruction}
  Let $U$ be a $K$-variety. Assume that 
  \begin{enumerate}[label = (\Alph*), start = 15]
    \item\label{enum:obstruction-assumption} there are nonconstant regular functions $s_1,\dots,s_n\in \cO_U(U)\setminus K$ and a constant $C>0$ such that, for every point
    $(x_v)_v\in \prod_{v\mid\infty} U(K_v)$, there is a $1\le i\le n$ with $\abs{s_i(x_v)}_v< C$ simultaneously for all $v\mid\infty$.
  \end{enumerate} Then $\mU(\fo_K)$ is not Zariski dense for any integral model $\mU$ of $U$.
\end{theorem}
\begin{proof}
  After multiplying the $s_i$ and $C$ with a suitable constant, we can assume that they are regular on $\mU$. Let $\alpha_1,\dots,\alpha_s$ be the finitely many integers in $\fo_K$ such that $\abs{\alpha_j}_v < C$ for all $v\mid\infty$. Then each point $P\in\mU(\fo_K)$ lies on one of the finitely many subvarieties
  $V(s_i-\alpha_j)\subset U$.
\end{proof}

Returning to the setting with $(X,D)$ as in the beginning of this section (assumptions that are stronger than necessary for what follows), Lemma~\ref{lem:obstruction-vanishes} and the following Lemma~\ref{lem:O-implies-1-dim-obs} will imply that this obstruction, too, vanishes after a suitable finite base change.

\begin{lemma}\label{lem:O-implies-1-dim-obs}
  If~\ref{enum:obstruction-assumption} holds, then for all $\uA\in \Can_\infty(D)$ such that each $A_v$ is of dimension $0$---that is, $A_v =\{D_v\}$ for some component $D_v$ of $D$---the group $\cO_{X}(U_{\uA})$ is nontrivial.
\end{lemma}
\begin{proof}
  For each $s_i$, consider the set $U_i = \{\abs{s_i(x_v)}_v < C \text{ for all $v$}\} \subset X(K_\RR)$. By assumption, these sets cover $U(K_\RR)$, and as this cover is finite, their closures $F_i$ cover $X(K_\RR)$.
  Let $\uA\in\Can_\infty(D)$ be such that each $A_v$ is of dimension $0$. For each $i$, write the divisor associated with $s_i$ as a rational function on $X$ as a difference $\sdiv s_i = M_i-N_i$ of effective divisors (without shared components).
  For each $v\mid\infty$, write
  \[
    D_v'=D_v(K_v)\setminus \bigcup_{D_v\not\subset M_i} M_i(K_v),
  \]
  which is nonempty by the smoothness of $D_v$.

  If an $s_i$ is not in $\cO_{X}(U_{\uA})$, that is, if it has a pole along at least one of the $D_v$, then for that $v\mid\infty$, the set
  $U_{i,v} = \{\abs{s_i(x_v)}_v < C\}$ is disjoint from $D_v(K_v)$ in $X(K_v)$. Its closure $F_{i,v}$ can only meet $D_v(K_v)\subset N_i(K_v)$ along the base locus of the pencil spanned by $M_i$ and $N_i$. (Indeed, blowing up the base locous yields a variety $\pi\colon X'\to X$ with a morphism $\sigma\colon X'\to\PP^1$ whose fibers are the strict transforms of the divisors in the pencil; then $F_{i,v}$ is $\pi(\sigma^{-1}(B))$ for the closed ball $B$ of radius $C$ around $0$ by properness of $\pi$, while $D_v(K_v)$ is contained in $N_i=\pi(\sigma^{-1}(\infty))$.)
  It follows that $F_{i,v}$ is disjoint from $D_v'$, so that $F_i$ is disjoint from $\prod_v D_v'$.

  If none of the $s_i$ were in $\cO_{X}(U_{\uA})$, then $\bigcup_i F_i = X(K_\RR)$ would be disjoint from the nonempty subset $\prod_v D_v'$, a contradiction.
\end{proof}

\begin{lemma}
  There is a finite field extension $L/K$ such that~\ref{enum:obstruction-assumption} does not hold for $U_L$.
\end{lemma}
\begin{proof}
  By Lemma~\ref{lem:obstruction-vanishes}, there is a field extension such that $\cO_{X_L}((U_L)_{\uB}) = L$ for at least one maximal face $\uB\in\Can_\infty(D_L)$. 
  If $D=0$, then $U=X$ is proper and $\cO_X(X)=K$, so~\ref{enum:obstruction-assumption} cannot hold. Otherwise, possibly enlarging $L$ to a totally imaginary field to make all $K_v$-analytic Clemens complexes coincide with the geometric one, $\uB$ contains a subface $\uA$ such that all $A_v$ have dimension $0$; as $U_\uA\supset U_{\uB}$, group of regular functions $\cO_{X_L}((U_L)_\uA)\subset \cO_{X_L}((U_L)_\uB)$ is still trivial.
  Now the statement follows from Lemma~\ref{lem:O-implies-1-dim-obs}.
\end{proof}

\subsection{Relating the obstruction to the constants \texorpdfstring{$\alpha_\uA$}{alpha A}}\label{ssec:relation-obstruction-constant}

This obstruction can be triggered if some of the objects defined in the previous section behave pathologically.

\begin{theorem}\label{thm:obstruction-construction}
  Let $\uA\in\Canmax_\infty(D)$ be a maximal face such that one of the following holds:
  \begin{enumerate}
    \item\label{enum:nonconvex} the effective cone $\Eff_\uA$ is not strictly convex (that is, $\alpha_\uA = 0$),
    \item\label{enum:b_A-unequal} $b_\uA \neq b^\prime_\uA$, or
    \item\label{enum:torsion} the group $\Pic(U;\uA)$ of $\uA$-divisor classes is not torsion free.
  \end{enumerate}
  Then $\cO_X(U_\uA)\ne K$, and there is an obstruction to the Zariski density of points near $\uA$.
  Moreover, if $K$ has only one infinite place and~(\ref{enum:b_A-unequal}) holds, then $\mU(\fo_K)$ is not Zariski dense for any integral model $\mU$ of $U$.
\end{theorem}
\begin{proof}
  Case~\eqref{enum:nonconvex}.
  That $\Eff_\uA$ is not strictly convex means that it contains a line through $0$, that is, we can find two nonzero effective divisors $(E, (E_v)_v)$ and $(E', (E'_v)_v)\in\Div(U;\uA)$ with $E+E^\prime\sim 0$. Hence, there exists a rational function which vanishes on all $E,E_v, E', E_v'$ (and thus is nonconstant), and whose only poles are outside $U_\uA$.

  \smallskip

  Case~\eqref{enum:b_A-unequal}.
  We have seen in Lemma~\ref{lem:descriptions-b_A} that $b_\uA \neq b^\prime_\uA$ if and only if there is a nonconstant invertible regular function $s\in E(U_\uA)$; so, in particular, $\cO_X(U_\uA)\ne K$ in this case.

  Next, assume that $K$ has only one infinite place, that is, that the group of units $\fo_K^\times$ is finite, and let $s\in E(U_\uA)$ be such an invertible regular function. After multiplying $s$ and $s^{-1}$ with appropriate constants, we get regular sections $s$ and $s^\prime$ on $\mU$ such that $ss^\prime=a\in\fo_K$. For a rational point $x\in\mU(\fo_K)$, the value $s(x)$ then has to be a divisor of $a$, of which there are only finitely many.
  The integral point $x$ must thus lie on one of the finitely many subvarieties $\overline{V(s-\alpha)}_{\alpha\mid a}$ of $X$.

  \smallskip
  
  Case~\eqref{enum:torsion}. Consider the embedding
  \[
    \ZZ^{\cA \,\setminus\, \supp{\Delta_\uA}} \to \bigoplus_v\ZZ^{A_v}, \quad
    L\mapsto (L|_{D_{A_v}})_v;
  \]
  let $M$ be the quotient, and observe that it is torsion free. Let $\phi\colon \Pic(U;\uA) \to M$ map the class of $(L, (L_v)_v)$ to the class of $(L_v)_v$; indeed, this is well-defined as any divisor of the form $\sdiv_\uA(f)$ maps to 
  \[
    \left(\sum_{\alpha \not\in\supp\Delta_\uA} \ord_{D_\alpha}(f) D_\alpha \bigg|_{D_{A_v}}\right)_v,
  \]
  which has trivial class in $M$. Then the sequence
  \begin{equation}
    \begin{tikzcd}
    \CH^0(\Delta_\uA) \arrow[r, "(i_{\Delta_\uA})_*"] &
    \Pic(X) \arrow[r, "\pi_\uA"] &
    \Pic(U;\uA) \arrow[r, "\phi"] &
    M \arrow[r] &
    0,
    \end{tikzcd}
  \end{equation}
  where $i_{\Delta_\uA}\colon \Delta_\uA \to X$ is the inclusion, is exact.
  Indeed, the kernel of $\pi_\uA$ is generated by divisors supported on $\Delta_\uA$, hence by the image of the pushforward map $\CH^0(\Delta_\uA)\to \Pic(X)$;
  for exactness on the right, note that $\Pic(X)=(\Div(U)\oplus\ZZ^\cA)/\im(\sdiv_X)$, so the cokernel of $\pi_\uA$ is indeed $\bigoplus_v\ZZ^{A_v}/\ZZ^{\cA\,\setminus\,\supp\Delta_\uA}$, after omitting the part $\ZZ^{\supp \Delta_\uA}$ mapped to $0$ by $\pi_\uA$.

  It follows that every nonzero torsion element $T\in\Pic(U;\uA)$ has to be the image of a (nonzero) element
  $\widetilde T\in\Pic(X)$ such that
  $n \widetilde{T} \in \im(\ZZ^{\supp{\Delta_\uA}})$;
  that is, there are $b_\alpha\in\ZZ$ such that $n \widetilde{T}+\sum b_\alpha D_\alpha \sim 0$. Consider
  \[
    \widetilde{T}^\prime = T + \sum \left\lceil\frac{b_\alpha}{n}\right\rceil D_\alpha.
  \]
  The divisor $\widetilde{T}^\prime$ is nonzero and in the pseudoeffective cone, so, using our assumptions on $X$, it is represented by an effective $\QQ$-divisor $E$.
  The image of $\widetilde{T}^\prime=[E]$ is still $T$, so the image of $[nE]$ is trivial. Working with a suitable multiple of $nE$ that is integral, this means that there is a rational function $s$ that vanishes on the support of $E$ and can only have poles on $\Delta_\uA$. Since the image of $[E]$ in $\Pic(U;\uA)$ is nonzero, the support of $E$ cannot be contained in the support of $\Delta_\uA$. Hence, $s$ is nonconstant and regular on $U_\uA$, and $\cO_X(U_\uA)\ne K$.
\end{proof}

The toric variety studied in Section~\ref{sec:toric} furnishes an example for Theorem~\ref{thm:obstruction-construction}~\eqref{enum:nonconvex}. For the other two cases, we have the following:

\begin{example}\label{ex:unitequation}
  Consider $\PP^n$ over a number field $K$ with $r$ real and $s$ complex places, together with the three hyperplanes $V(x_0)$, $V(x_1)$, and $V(x_0+x_1)$. Their sum does not have strict normal crossings, a situation that can be remedied by blowing up $V(x_0,x_1)$. Call the resulting variety $X$, and consider the pair $(X,D)$ with $D=H_1+H_2+H_3+E$, where the $H_i$ are the strict transforms of the three hyperplanes and $E$ is the exceptional divisor. For $n\ge 3$, the log anticanonical bundle is big (though never nef), and we have
  \[
    U=X\setminus D \cong \bAA^n \setminus (V(x_1)\cup V(x_1+1)).
  \]
  The geometric and every $K_v$-analytic Clemens complex is a ``star'', with the vertex corresponding to $E$ connected to the other three vertices $H_i$. If we take $\uA=(A_v)_v$ with the same maximal face $A_v=A=\{E,H_i\}$ (for some fixed $i$) for all infinite places, we have $U_\uA\cong \bAA^{n-1}\times \multgrp$. Hence,
  \[
    b_\uA=\rk\Pic(U) - \rk E(U) + \sum_{v\mid\infty} \# A = 0 - 2 + 2(r+s)=2(r+s)-2.
  \]
  On the other hand, using~\eqref{eq:sequence-pic-u-a}, the sequence
  \[
    0\to E(U_\uA) \to E(U) \to (\ZZ^A)^{\oplus (r+s)} \to \Pic(U;\uA) \to 0
  \]
  is exact, with the groups to the left having ranks $1$, $2$, and $2(r+s)$, respectively, so
  $b^\prime_\uA=2(r+s)-1$, and there is an obstruction by Theorem~\ref{thm:obstruction-construction}~\eqref{enum:b_A-unequal}. In fact, the set of integral points is not Zariski dense: every integral point lies on one of the subvarieties $\{ax_0-bx_1=0\}$ parametrized by the finitely many solutions $a,b\in\fo_K^\times$ of the unit equation $a+b=1$.
  (Note that by Lemma~\ref{lem:obstruction-vanishes}, there are unobstructed faces over sufficiently large fields, so this failure of Zariski density is not explained by Corollary~\ref{cor:density-obstruction} in general.)
\end{example}

\begin{example}
  Consider $\PP^n$ together with a divisor $D$ having two components: the quadric hypersurface $Q=\{x_0^2=\sum_{i=1}^n x_i^2\}$ and the hyperplane $H=\{x_0=0\}$. If $n\ge 3$, the log anticanonical bundle is ample. The intersection $Q\cap H$ does not contain any $\RR$-points; so, if $K$ is a totally real field, every $K_v$-analytic Clemens complex consists of two isolated vertices. Consider the face $\uA=(H,\dots,H)\in\Canmax_v(D)$. Since the Picard group of $U=\PP^n\setminus D$ is trivial,~\eqref{eq:sequence-pic-u-a} allows us to compute
  \[
    \Pic(U;\uA) \cong \ZZ^r/(2,\dots,2) \cong \ZZ^{r-1} \oplus \ZZ/2\ZZ,
  \]
  and Theorem~\ref{thm:obstruction-construction}~\eqref{enum:torsion} applies.
  Note that if $K=\QQ$, there are only finitely many points corresponding to the solutions of $x_1^2 + \cdots + x_n^2 = 2$, while for larger fields, we get the sets of solutions of
  $x_1^2 + \cdots + x_n^2 = 1 + u$
  for units $u\in\fo_K^\times$.
\end{example}

\subsection{Asymptotic formulas}\label{ssec:asymptotic-formulas}
These definitions allow the interpretation of asymptotic formulas. Keep all the assumptions on $(X,D)$ from the beginning of this chapter, which included $X$ and $D$ being \emph{split}. Let $\mU$ be an integral model of $U$, and assume that $\mU(\fo_K)$ is not thin (whence in particular Zariski dense). Let $H$ be the height function associated with a metric on the log anticanonical bundle $\omega_X(D)^\vee$. We are interested in the asymptotic behavior of the number
\[
  N(B)=\{x\in\mU(\fo_K)\cap V \mid H(x) \le B\}
\]
of integral points of bounded height whose generic point lies on the complement $V$ of an appropriate accumulating thin set $Z \subset X(K)$. If strong approximation holds (using the set of connected components at archimedean places, cf.\ e.g.\ \cite{MR4288633}), asymptotic expansions for $N(B)$ tend to be similar to
\begin{equation}\label{eq:asymptotic-formulas}
  c_\infty c_{\fin} B (\log B)^{b-1}(1+o(1)),
\end{equation}
where
\begin{align*}
  c_\infty &= \frac{1}{\abs{d_K}^{\dim U/2}}
  \sum_{
    \uA\in\cC^{\tmax,\circ}(D)
  }
  \alpha_\uA \prod_{v\mid\infty}
  \tau_{Z_{A_v},v} \left(Z_{A_v}(K_v)\right) \qquad \text{and}
  \\
  c_{\fin} &=
  \rho_K^{\rk\Pic U - \rk E(U)}
  \prod_{v<\infty}
  \left(1-\frac{1}{\#k_v}\right)^{\rk\Pic U - \rk E(U)}
  \tau_{U,v}(\mU(\fo_{K_v})).
\end{align*}
Here, the number $b$ in the exponent of $\log B$ is the maximal value of $b_\uA=b_\uA^\prime$ attained on tuples $\uA$ of maximal faces with $\cO_X(U_\uA) = K$, i.e., on tuples without an obstruction. The sum runs over the set
\[
  \cC^{\tmax,\circ}(D) = \left\{\uA \in \Canmax_\infty(D) \relmiddle| \cO_X(U_\uA) = K,\ b_\uA=b \right\}
\]
of faces $\uA$ on which this maximum $b$ is attained, that is, the set of maximal dimensional faces under those without an obstruction. Corollary~\ref{cor:density-obstruction} guarantees that the sum does not run over the empty set, and Theorem~\ref{thm:obstruction-construction}~\eqref{enum:nonconvex} guarantees that the factors $\alpha_\uA$ are nonzero.

In a more general setting,
the product of volumes has to be replaced by the volume of a suitable subset of adelic points: those points in
\[
  \bigcup_{\uA \in \cC^{\tmax,\circ}(D)}
  \prod_{v<\infty}\mU(\fo_K) \times \prod_v Z_{A_v}(K_v)
\]
that are limit points of integral points (or, if there still are no such limit points, a similarly defined set using maximal faces of smaller dimension necessitating a further change of $b$), making necessary adjustments to the Tagamawa volume on $Z_{A_v}(K_v)$ if there are failures of strong approximation involving some connected components of $U(K_v)$ bordering on $Z_{A_v}(K_v)$ for archimedean $v$ to account for the fact that there are fewer points near this stratum.

Moreover, the factor $\rho_K$ has to be replaced by the principal value of a different $L$-function, and additional factors can appear in the constant, related to failures of strong approximation, to nonsplitness, and to cohomological invariants (similar to $\beta$ in the case of rational points). It is unclear to the author what the shape of such a factor for arbitrary $(X,D)$ should be and under which conditions one should expect it to be different from $1$. Note that the Brauer group modulo constants, whose order $\beta$ is a factor of Peyre's constant for rational points, might be nontrivial even for split~$U$.

We can compare~\eqref{eq:asymptotic-formulas} to results in the framework by Chambert-Loir and Tschinkel. We note a difference in the case of toric varieties, and list the additional factors appearing in these asymptotic formulas.
\begin{itemize}
  \item The formula above agrees with~\cite[Thm.~3.5.6]{MR2999313} on partial equivariant compactifications of vector groups, since the obstruction never occurs in these cases, and since the cones $\Eff_\uA$ are all smooth, satisfying
  \[
    \alpha_\uA=\frac{1}{(b-1)!} \left(\prod_{\alpha\not\in\cA} \frac{1}{\rho_\alpha} \right) \left( \prod_{v\mid\infty} \prod_{\alpha\in A_v} \frac{1}{\rho_\alpha -1} \right)
  \]
  with the description $-K_X=\sum_{\alpha\in\cD} \rho_\alpha D_\alpha$ of the anticanonical divisor as a sum of the boundary components $\{D_\alpha\}_{\alpha\in \cD}$.

  \item Similarly, in the case of partial equivariant compactifications of split semi\-simple groups $G$~\cite{MR3117310}, the obstruction does not occur, and the cones are smooth with a similar description of $\alpha_\uA$, making the formulas compatible. An additional factor is part of the asymptotic formula (18) in op.\ cit.:\ the number $\abs{\chi_{S,D,\lambda}(G)}$ of certain automorphic characters of the underlying group $G$, related to strong approximation on $G$.

  \item The formula~\eqref{eq:asymptotic-formulas} is not compatible with~\cite[Thm.~3.11.5]{arXiv:1006.3345} on toric varieties; it modifies the exponent $b-1$ of $\log B$ and the index set of the sum. Our formula above agrees with the asymptotic formula we determine in Section~\ref{sec:toric}.
  The formula in loc.~cit.\ contains additional factors
  \[
    \frac{\abs{A(T,U,K)^*}}{\abs{A(T)^*}}
    \frac{ \abs{H^1(\Gamma,\Pic(X_E))} }{ \abs{H^1(\Gamma,M_E)} }:
  \]
  two groups of automorphic characters, related to weak and strong approximation on $T$, and cohomology groups from the action of the Galois group (which is trivial in the split case). Moreover, the volume is taken on the subset of the adelic points cut out by these automorphic characters.

  \item The formula is compatible with~\cite{arxiv:2109.06778}, treating integral points of several open subvarieties of the minimal desingularization of a singular quartic del Pezzo surface. This variety is an example of a nontoric variety in which the construction of $\alpha_\uA$ does not lead to a simplicial cone.
\end{itemize}

\section{Integral points on a toric threefold}\label{sec:toric}

The aim of this section is to provide an asymptotic formula for the number $N(B)$ of integral points of height at most $B$ on the toric variety $X$ defined in the introduction. Integral points on toric varieties are treated by Chambert-Loir and Tschinkel in~\cite{arXiv:1006.3345}; however, our result contradicts part of this (unfinished) work.
After parametrizing the set of integral points using a universal torsor in Section~\ref{ssec:toric-torsor}, we determine an asymptotic formula in Section~\ref{ssec:toric-counting}, proving Theorem~\ref{thm:intro-count}. The exponent of $\log B$ is $1$ less than the one given in~\cite{arXiv:1006.3345}, which is explained by an obstruction to the existence of integral points on a certain part of $X$: Chambert-Loir's and Tschinkel's asymptotic formula is associated with the one-dimensional face $\{E_1,E_2\}$ of the Clemens complex. There is a function obstructing the Zariski density of integral points near $E_1$ and $E_2$, which also makes the leading constant of their asymptotic formula vanish.
In Section~\ref{ssec:toric-interpretation}, we compare our formula to the one given by Chambert-Loir and Tschinkel in greater detail and get a very similar geometric interpretation to theirs (Theorem~\ref{thm:interpreted-count}), associated with the maximal, but only zero-dimensional face $M$ of the Clemens complex.
\begin{theorem}\label{thm:interpreted-count}
  The number of integral points of bounded height satisfies the asymptotic formula
  \[
    N(B) = c_{\infty} c_{\fin}B (\log B)^{b_M-1} (1+o(1)),
  \]
  with
  \begin{align*}
    c_{\infty} &=\alpha_M \tau_{M,\infty}(M(\RR)),\\
    c_{\fin} &= \prod_p \left(1-\frac{1}{p}\right)^{\rk\Pic U} \tau_{U,p}(\mU(\ZZ_p)),
  \end{align*}
  where all constants are associated with the maximal, but not maximal-dimensional, face $M$ of the Clemens complex.
  More explicitly,
  \[
    N(B)=c B (\log B)^2 + O(B\log B(\log\log B)^3)\text{,}
  \]
  where
  \[
    c=4 \prod_p \left(\left(1-\frac{1}{p}\right)^2\left(1+\frac{2}{p}-\frac{1}{p^2}-\frac{1}{p^3}\right)\right) \text{.}
  \]
\end{theorem}

\subsection{Passage to a universal torsor}\label{ssec:toric-torsor}

The fan $\Sigma_X$ of $X$ can be obtained by starting with the fan of $\PP^1\times\PP^1\times\PP^1$, then subdividing it by adding the ray $\rho_x=\RR(-1,-1,0)$ (corresponding to the exceptional divisor $E_1$), then further subdividing it by adding the ray $\rho_y=\RR(-1,0,-1)$ (corresponding to $E_2$).
The Picard group of $X$ is
\[
  \Pic(X)=\ZZ \pi^*[H_1] + \ZZ \pi^*[H_2] + \ZZ \pi^*[H_3] +\ZZ [E_1] + \ZZ [E_2] \cong \ZZ^5 \text{,}
\]
where $H_1$, $H_2$, and $H_3$ are planes of degree $(1,0,0)$, $(0,1,0)$, and $(0,0,1)$, respectively.

\begin{figure}[ht]
  \begin{center}
\tdplotsetmaincoords{250}{170}
\begin{tikzpicture}[tdplot_main_coords]
  \def\pcolor{gray!90}
  \def\popacity{0.1}
  \def\dstyle{dotted}

  \draw[thick] (0,0,0) -- (2,0,0)   node[anchor=east]{$a_0$};
  \draw[thick] (0,0,0) -- (0,2,0)   node[anchor=south]{$b_0$};
  \draw[thick] (0,0,0) -- (0,0,2)   node[anchor=north]{$c_0$};
  \draw[thick] (0,0,0) -- (-2,0,0)  node[anchor=west]{$a_1$};
  \draw[thick] (0,0,0) -- (0,-2,0)  node[anchor=north]{$b_1$};
  \draw[thick] (0,0,0) -- (0,0,-2)  node[anchor=south]{$c_1$};
  \draw[thin] (0,0,0) -- (-1,-1,0) node[anchor=north west]{$x$};
  \draw[thin] (0,0,0) -- (-1,0,-1) node[anchor=south west]{$y$};

  \path[fill=\pcolor, opacity=\popacity] (0,0,0) -- (2,0,0) -- (0,2,0);
  \draw[\dstyle]                                    (2,0,0) -- (0,2,0);
  \path[fill=\pcolor, opacity=\popacity] (0,0,0) -- (2,0,0) -- (0,0,2);
  \draw[\dstyle]                                    (2,0,0) -- (0,0,2);
  \path[fill=\pcolor, opacity=\popacity] (0,0,0) -- (0,2,0) -- (0,0,2);
  \draw[\dstyle]                                    (0,2,0) -- (0,0,2);

  \path[fill=\pcolor, opacity=\popacity] (0,0,0) -- (2,0,0) -- (0,2,0);
  \draw[\dstyle]                                    (2,0,0) -- (0,2,0);
  \path[fill=\pcolor, opacity=\popacity] (0,0,0) -- (2,0,0) -- (0,0,-2);
  \draw[\dstyle]                                    (2,0,0) -- (0,0,-2);
  \path[fill=\pcolor, opacity=\popacity] (0,0,0) -- (0,2,0) -- (0,0,-2);
  \draw[\dstyle]                                    (0,2,0) -- (0,0,-2);

  \path[fill=\pcolor, opacity=\popacity] (0,0,0) -- (2,0,0) -- (0,-2,0);
  \draw[\dstyle]                                    (2,0,0) -- (0,-2,0);
  \path[fill=\pcolor, opacity=\popacity] (0,0,0) -- (2,0,0) -- (0,0,2);
  \draw[\dstyle]                                    (2,0,0) -- (0,0,2);
  \path[fill=\pcolor, opacity=\popacity] (0,0,0) -- (0,-2,0) -- (0,0,2);
  \draw[\dstyle]                                    (0,-2,0) -- (0,0,2);

  \path[fill=\pcolor, opacity=\popacity] (0,0,0) -- (2,0,0) -- (0,-2,0);
  \draw[\dstyle]                                    (2,0,0) -- (0,-2,0);
  \path[fill=\pcolor, opacity=\popacity] (0,0,0) -- (2,0,0) -- (0,0,-2);
  \draw[\dstyle]                                    (2,0,0) -- (0,0,-2);
  \path[fill=\pcolor, opacity=\popacity] (0,0,0) -- (0,-2,0) -- (0,0,-2);
  \draw[\dstyle]                                    (0,-2,0) -- (0,0,-2);

  \path[fill=\pcolor, opacity=\popacity] (0,0,0) -- (-2,0,0) -- (0,2,0);
  \draw[\dstyle]                                    (-2,0,0) -- (0,2,0);
  \path[fill=\pcolor, opacity=\popacity] (0,0,0) -- (-2,0,0) -- (0,0,2);
  \draw[\dstyle]                                    (-2,0,0) -- (0,0,2);
  \path[fill=\pcolor, opacity=\popacity] (0,0,0) -- (0,2,0) -- (0,0,2);
  \draw[\dstyle]                                    (0,2,0) -- (0,0,2);

  \path[fill=\pcolor, opacity=\popacity] (0,0,0) -- (-2,0,0) -- (0,2,0);
  \draw[\dstyle]                                    (-2,0,0) -- (0,2,0);
  \path[fill=\pcolor, opacity=\popacity] (0,0,0) -- (-2,0,0) -- (-1,0,-1);
  \draw[\dstyle]                                    (-2,0,0) -- (-1,0,-1);
  \path[fill=\pcolor, opacity=\popacity] (0,0,0) -- (-1,0,-1) -- (0,0,-2);
  \draw[\dstyle]                                    (-1,0,-1) -- (0,0,-2);
  \path[fill=\pcolor, opacity=\popacity] (0,0,0) -- (0,2,0) -- (0,0,-2);
  \draw[\dstyle]                                    (0,2,0) -- (0,0,-2);
  \path[fill=\pcolor, opacity=\popacity] (0,0,0) -- (0,2,0) -- (-1,0,-1);
  \draw[\dstyle]                                    (0,2,0) -- (-1,0,-1);

  \path[fill=\pcolor, opacity=\popacity] (0,0,0) -- (-2,0,0) -- (-1,-1,0);
  \draw[\dstyle]                                    (-2,0,0) -- (-1,-1,0);
  \path[fill=\pcolor, opacity=\popacity] (0,0,0) -- (-1,-1,0) -- (0,-2,0);
  \draw[\dstyle]                                    (-1,-1,0) -- (0,-2,0);
  \path[fill=\pcolor, opacity=\popacity] (0,0,0) -- (-2,0,0) -- (0,0,2);
  \draw[\dstyle]                                    (-2,0,0) -- (0,0,2);
  \path[fill=\pcolor, opacity=\popacity] (0,0,0) -- (0,-2,0) -- (0,0,2);
  \draw[\dstyle]                                    (0,-2,0) -- (0,0,2);
  \path[fill=\pcolor, opacity=\popacity] (0,0,0) -- (-1,-1,0) -- (0,0,2);
  \draw[\dstyle]                                    (-1,-1,0) -- (0,0,2);

  \path[fill=\pcolor, opacity=\popacity] (0,0,0) -- (-2,0,0) -- (-1,-1,0);
  \draw[\dstyle]                                    (-2,0,0) -- (-1,-1,0);
  \path[fill=\pcolor, opacity=\popacity] (0,0,0) -- (-1,-1,0) -- (0,-2,0);
  \draw[\dstyle]                                    (-1,-1,0) -- (0,-2,0);
  \path[fill=\pcolor, opacity=\popacity] (0,0,0) -- (-2,0,0) -- (-1,0,-1);
  \draw[\dstyle]                                    (-2,0,0) -- (-1,0,-1);
  \path[fill=\pcolor, opacity=\popacity] (0,0,0) -- (-1,0,-1) -- (0,0,-2);
  \draw[\dstyle]                                    (-1,0,-1) -- (0,0,-2);
  \path[fill=\pcolor, opacity=\popacity] (0,0,0) -- (0,-2,0) -- (0,0,-2);
  \draw[\dstyle]                                    (0,-2,0) -- (0,0,-2);
  \path[fill=\pcolor, opacity=\popacity] (0,0,0) -- (-1,-1,0) -- (0,0,-2);
  \draw[\dstyle]                                    (-1,-1,0) -- (0,0,-2);
  \path[fill=\pcolor, opacity=\popacity] (0,0,0) -- (-1,-1,0) -- (-1,0,-1);
  \draw[\dstyle]                                    (-1,-1,0) -- (-1,0,-1);
\end{tikzpicture}
    \caption[The fan $\Sigma_X$ of $X$]{The fan $\Sigma_X$ of $X$, its rays labeled with the corresponding generators of the Cox ring.}
  \end{center}
\end{figure}
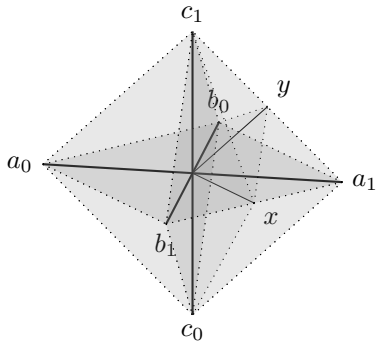

Aiming to count integral points by a parametrization using a universal torsor, we start by determining the Cox ring; cf.\ for example~\cite[\S\,2.1.3]{MR3307753} for background on the following constructions.
The Cox ring of $X$ is $R_X=\QQ[a_0,a_1,b_0,b_1,c_0,c_1,x,y]$, its generators corresponding to the rays of $\Sigma_X$. It is graded by $\Pic(X)$, the degree of each generator being the class of the corresponding divisor in the Picard group.  Under the above isomorphism $\Pic(X)\cong \ZZ^5$, the grading is thus given by Table~\ref{tab:degrees}.
\begin{table}[ht]
  \begin{tabular}{cccccccc}\toprule 
    $a_0$ & $a_1$ & $b_0$ & $b_1$ & $c_0$ & $c_1$ &   $x$ &   $y$ \\ \midrule
      $1$ &   $1$ &   $0$ &   $0$ &   $0$ &   $0$ &   $0$ &   $0$ \\
      $0$ &   $0$ &   $1$ &   $1$ &   $0$ &   $0$ &   $0$ &   $0$ \\
      $0$ &   $0$ &   $0$ &   $0$ &   $1$ &   $1$ &   $0$ &   $0$ \\
      $0$ &  $-1$ &   $0$ &  $-1$ &   $0$ &   $0$ &   $1$ &   $0$ \\
      $0$ &  $-1$ &   $0$ &   $0$ &   $0$ &  $-1$ &   $0$ &   $1$ \\ \bottomrule
  \end{tabular}
  \caption{The generators of $R_X$ and their degrees in $\ZZ^5 \cong \Pic(X)$.}\label{tab:degrees}
\end{table}

The irrelevant ideal is generated by the set $\{\prod g \mid \rho_g \not\subset \sigma\}_{\sigma\in\Sigma^{\text{(max)}}}$; it is thus
\begin{align*}
  I_\irr=(
  &a_1b_1c_1xy,\,
  a_1b_0c_1xy,\,
  a_1b_1c_0xy,\,
  a_1b_0c_0xy,\\
  &a_0b_1c_1xy,\,
  a_0b_0b_1c_1y,\,
  a_0a_1b_0c_1y,\,
  a_0b_1c_0c_1x,\\
  &a_0a_1b_1c_0x,\,
  a_0a_1b_0c_0y,\,
  a_0b_0b_1c_0c_1,\,
  a_0a_1b_0b_1c_0
  ),
\end{align*}
and we get a universal torsor $Y=\Spec R_X \setminus V(I_\irr) \to X$. 
The image of a point 
\[
  (a_0,a_1,b_0,b_1,c_0,c_1,x,y)\in Y(\QQ)
\]
is denoted by $\pc{a_0:a_1:b_0:b_1:c_0:c_1:x:y}\in Y(\QQ)$ (expressed in \emph{Cox coordinates}), and is further mapped to $(\pc{a_0:a_1xy},\, \pc{b_0:b_1x},\, \pc{c_0:c_1y})\in X_0(\QQ)$ by the blow-up morphism $\pi$.

\begin{lemma}\label{lem:description-as-difference}
  The log anticanonical bundle is big, i.e., in the interior of the effective cone, but it is not nef. It has the description $\omega_X(D)^\vee\cong \cL_1 \otimes \cL_2^\vee$ as a quotient of base point free bundles, where the class of $\cL_1$ is $(2,2,2,-2,-2)$, and the class of $\cL_2$ is $(1,0,0,0,0)$ under the above isomorphism $\Pic(X)\cong\ZZ^5$.
\end{lemma}
\begin{proof}
  In Cox coordinates, the exceptional divisors are defined by $E_1=V(x)$ and $E_2=V(y)$ and the third component of the boundary is $M=V(a_0)$.
  The log anticanonical class $\omega_X(D)$ corresponds to
  \[
    (1,2,2,-2,-2)=\sum_{g\text{ generator of } R_X} \deg(g) - \deg(x)-\deg(y)-\deg(a_0)
  \]
  under the above isomorphism $\Pic(X)\cong\ZZ^5$.
  It is not base point free, since $b_1c_1$ divides all of its global sections. Since the same holds for all its multiples, it is not semi-ample and, as a consequence, not nef, since the two notions coincide on toric varieties. It is, however, big: the effective cone is generated by the degrees of the generators of the Cox ring, and 
  \begin{equation*}
    (1,2,2,-2,-2) = \tfrac{\deg(a_0) + 3 \deg(a_1) + \deg(b_0) + 7\deg(b_1) + \deg(c_0) + 7\deg(c_1) + 2\deg(x) + 2\deg(y)}{4}
  \end{equation*}
  is in its interior. We shall use its description $(2,2,2,-2,-2)-(1,0,0,0,0)$ as a difference of base point free classes to construct a corresponding height function. The sets
  \begin{equation}\label{eq:sections}
  \{a_1^2b_0^2c_0^2,
    a_1^2b_1^2c_0^2x^2,
    a_1^2b_0^2c_1^2y^2,
    a_1^2b_1^2c_1^2x^2y^2,
    a_0^2b_1^2c_1^2
  \} \text { and } \{
    a_0, a_1xy
  \}
  \end{equation}
  of elements of the Cox ring correspond to global sections of $\cL_1$ and $\cL_2$, respectively. Neither of these sets can vanish simultaneously, so both classes are indeed base point free.
\end{proof}

These choices of sections induce metrics on the bundles $\cL_1$, $\cL_2$, and $\cL_1\otimes \cL_2^\vee \cong \omega_X(D)^\vee$, which in turn induce a log anticanonical height function.

\begin{lemma}\label{lem:4-to-1-and-height}
  There is a $4$-to-$1$-correspondence between the set of integral points $\mU(\ZZ)\cap T(\QQ)$ and the set
  \begin{equation*}
  \{(1,a_1,b_1,b_2,c_1,c_2,1,1)\in\ZZnz^8 \mid \text{\eqref{eq:gcd} holds} \} \subset Y(\QQ),
  \end{equation*}
  where 
  \begin{equation}\label{eq:gcd}
    \gcd(a_1b_0c_0,a_1b_0c_1,a_1b_1c_0,b_1c_1)=1.
  \end{equation}
  The log anticanonical height of the image of a point $(1,a_1,b_0,b_1,c_0,c_1,1,1)$ in the above set is
  \begin{equation}\label{eq:height}
    H(a_1,b_0,b_1,c_0,c_1)=\abs{a_1} \max\{\abs{b_0^2},\abs{b_1^2}\} \max\{\abs{c_0^2},\abs{c_1^2}\}\text{.}
  \end{equation}
\end{lemma}
\begin{proof}
  Consider an integral point as in the description~\eqref{eq:integral-points-intro}, that is, a point
  \[
    P=\left((a_0:a_1),\, (b_0:b_1),\, (c_0:c_1)\right) \in \mU(\ZZ)\subset X_0(\QQ)
  \]
  with $a_0\in\{\pm 1\}$ satisfying coprimality conditions that can be checked to be equivalent to~\eqref{eq:gcd}.
  Multiplying the first pair with $a_0$ eliminates the choice of sign in $a_0$; but multiplying any of the latter two pairs with a unit does not change the integral point $P$, resulting in the claimed $4$-to-$1$-correspondence. The point $P$ is the image of
  \begin{equation}\label{eq:int-point-on-torsor}
    (1,a_1,b_0,b_1,c_0,c_1,1,1)\in Y(\QQ)
  \end{equation}
  by the description of the torsor and blow-up morphism before Lemma~\ref{lem:description-as-difference}.
  Such a point is in the open torus $T(\QQ)$ if and only if $a_1b_0b_1c_0c_1\neq 0$.

  The choice~\eqref{eq:sections} of global sections of $\cL_1$ and $\cL_2$ induces metrics on these line bundles and, consequently, on the line bundle $\cL_1 \otimes \cL_2^\vee$ isomorphic to the log anticanonical bundle. The latter metric then induces a log anticanonical height function. Its value on the image of a point $(a_0,a_1,b_0,b_1,c_0,c_1,x,y)\in Y(\QQ)$ is
  \[
    \prod_{v}\ 
    \frac{\max
    \left\{\abs{a_1^2b_0^2c_0^2}_v,
      \abs{a_1^2b_1^2c_0^2x^2}_v,
      \abs{a_1^2b_0^2c_1^2y^2}_v,
      \abs{a_1^2b_1^2c_1^2x^2y^2}_v,
      \abs{a_0^2b_1^2c_1^2}_v
    \right\}}{
      \max\left\{\abs{a_0}_v, \abs{a_1xy}_v\right\}
    }\text{.}
  \]
  For a point as in~\eqref{eq:int-point-on-torsor} and a finite place $p$, the denominator is $\abs{a_0}_p = 1$, while the numerator is $1$ as a consequence of the coprimality condition~\eqref{eq:gcd}; hence, the factors associated with finite primes are $1$.
  Using that $\abs{a_0}=\abs{x}=\abs{y}=1$ and $\abs{a_1}\ge 1$, the factor at the archimedean place can be simplified to~\eqref{eq:height}. 
\end{proof}

\begin{remark}
  More formally, the fan $\Sigma_X$ also induces a \emph{toric $\ZZ$-scheme} $\mX$ and a universal torsor $\mY\to\mX$ (cf.~\cite[p.~187 and Rem.~8.6~(b)]{MR1679841}, building on~\cite{MR0284446}). The fiber above each integral point is a $\GmZ^5$-torsor, the isomorphism classes of which are parametrized by $H^1_{\fppf}(\Spec \ZZ; \GmZ^5) = \Cl(\QQ)^5$; as this group is trivial, each fiber is isomorphic to $\GmZ^5$, which has $2^5$ integral points. The set~\eqref{eq:int-point-on-torsor} of points in $Y(\QQ)$ then coincides with $(\mY\setminus V(a_0xy))(\ZZ)$ up to fixing the three signs $a_0,x,y\in\{\pm 1\}$, yielding a $4$-to-$1$-correspondence to integral points on the model $\mX\setminus \overline{D}$ of $U$, which can be checked to coincide with~$\mU$.
\end{remark}

\subsection{Counting}\label{ssec:toric-counting}
In other words, we now have a new description
\[
  N(B)=\frac{1}{4}\#\{(a_1,b_0,b_1,c_0,c_1)\in\ZZnz^5 \mid H(a_1,b_0,b_1,c_0,c_1)\leq B \text{, (\ref{eq:gcd}) holds}\}
\]
of the counting function, with the height function $H$ in~\eqref{eq:height}.
\begin{lemma}
  We have
  \[
    N(B)=\prod_p\left(\left(1-\frac{1}{p}\right)^2\left(1+\frac{2}{p}-\frac{1}{p^2}-\frac{1}{p^3}\right)\right) V(B)+O(B \log B (\log\log B)^3)
  \]
  with
  \[
    V(B)=\frac{1}{4}
    \int_{\substack{\abs{a_1},\abs{b_0},\abs{b_1},\abs{c_0},\abs{c_1} \geq 1,\\\abs{a_1}\max\{\abs{b_0^2},\abs{b_1^2}\}\max\{\abs{c_0^2},\abs{c_1^2}\}\leq B}}
    \diff a_1 \diff b_0 \diff b_1 \diff c_0 \diff c_1\text{.}
  \]
\end{lemma}
\begin{proof}
  The counting problem can be rephrased as
  \[
    N(B)=\frac{1}{4}\sum_{\substack{a_1,b_0,b_1,c_0,c_1\in\ZZnz\\H(a_1,b_0,b_1,c_0,c_1)\leq B}} \theta(a_1,b_0,b_1,c_0,c_1) \text{,}
  \]
  where $\theta = \delta_{\gcd(a_1b_0c_0,a_1b_0c_1,a_1b_1c_0,b_1c_1)=1}=\prod_p \theta^{(p)}$ with
  \[
    \theta^{(p)}(a_1,b_0,b_1,c_0,c_1)=\begin{cases}
      0, & \text{if } p\mid a_1b_0c_0 , a_1b_0c_1 , a_1b_1c_0 , b_1c_1, \\
      1, & \text{else.}
    \end{cases}
  \]
  Aiming to first replace the sum over $b_0$ by an integral, observe that
  the height conditions imply that
  \[
  \abs{a_1b_0^2c_0^2}, \abs{a_1b_1^2c_0^2}, \abs{a_1b_0b_1c_1^2}\leq B\text{,}
  \]
  since the latter one is the geometric average of two terms in the height function. We have
  \begin{align*}
    1= \frac{B}{\abs{a_1b_0b_1c_0c_1}}
    \left(\frac{B}{\abs{a_1b_0^2c_0^2}}\right)^{-1/4}
    \left(\frac{B}{\abs{a_1b_1^2c_0^2}}\right)^{-1/4}
    \left(\frac{B}{\abs{a_1b_0b_1c_1^2}}\right)^{-1/2},
  \end{align*}
  and note that the function $\theta$
  satisfies Definition 7.9 in~\cite{MR2520770}.
  Using~\cite[Prop.~3.9]{MR2520770} with $r=1$ and $s=3$, we get
  \[
    N(B)=\sum_{\substack{a_1,b_1,\\c_0,c_1\in\ZZnz}} \theta_1(a_1,b_1,c_0,c_1) V_1(a_1,b_1,c_0,c_1;B) +O(B\log B(\log\log B)^3)\text{,}
  \]
  where $V_1(a_1,b_1,c_0,c_1;B) = \frac{1}{4}\int_{\substack{\abs{b_0}\geq 1 \\ H(a_0,b_0,b_1,c_0,c_1)\leq B}} \diff b_0$
  and $\theta_1=\prod_p\theta_1^{(p)}$ with
  \[
    \theta_1^{(p)}(a_1,b_1,c_0,c_1) = \begin{cases}
      0, & \text{ if } p \mid  a_1 c_0, a_1c_1 , b_1c_1, \\
      1-\frac{1}{p}, & \text{ if } p \mid b_1, p\nmid a_1 \text{ and } (p\nmid c_0 \text{ or } p\nmid c_1),\\
      1, & \text{ if } p \nmid b_1 \text{ and } (p\nmid c_1 \text{ or } p\nmid a_1c_0) \text{.}
    \end{cases}
  \]
  Using the geometric average of the two height conditions involving $b_0$, we can bound $V_1$ by
  \[
    V_1(a_1,b_1,c_0,c_1;B) \ll \sqrt{\frac{B}{\abs{a_1c_0c_1}}} = \frac{B}{\abs{a_1b_1c_0c_1}}
    \left(\frac{B}{\abs{a_1b_1^2c_0^2}}\right)^{-1/4}
    \left(\frac{B}{\abs{a_1b_1^2c_1^2}}\right)^{-1/4}\text{.}
  \]
  Since $\abs{a_1b_1^2c_0^2}$ and $\abs{a_1b_1^2c_1^2}$ are bounded by $B$, applying~\cite[Prop.~3.9]{MR2520770} once more (with $r=1$, $s=2$) yields
  \[
    N(B)=\sum_{a_1,b_1,c_1\in\ZZnz} \theta_2(a_1,b_1,c_1) V_2(a_1,b_1,c_1;B) +O(B\log B(\log\log B)^3)\text{,}
  \]
  where $V_2(a_1,b_1,c_1;B) = \frac{1}{4}\int_{\substack{\abs{b_0},\abs{c_0}\geq 1 \\ H(a_0,b_0,b_1,c_0,c_1)\leq B}} \diff(b_0,c_0)$
  and $\theta_2=\prod_p\theta_2^{(p)}$ with
  \[
    \theta_2^{(p)}(a_1,b_1,c_0,c_1) = \begin{cases}
      0, & \text{if } p \mid a_1, b_1 c_1\\
      \left(1-\frac{1}{p}\right)^2,
        & \text{if } p\mid b_1c_1 \text{, } p\nmid a_1 \\
      1-\frac{1}{p},
        & \text{if } p \mid b_1\text{, } p\nmid a_1c_1 \\
      1-\frac{1}{p},
          & \text{if } p \mid c_1\text{, } p\nmid a_1b_1 \\
      1, & \text{if } p \nmid b_1c_1 \text{.}
    \end{cases}
  \]
  To complete the summations, we use the fact that the height conditions imply $\abs{a_1b_0^2c_0c_1}\leq B$, and get an upper bound
  \begin{align*}
    V_2(a_1,b_1,c_1;B) &\ll \int_{\substack{\abs{c_0}\geq 1 \\ \abs{a_1b_1^2c_0^2}\leq B}}  \sqrt{\frac{B}{\abs{a_1c_0c_1}}} \diff c_0  \ll \frac{B^{3/4}}{\abs{a_1}^{3/4}\abs{b_1}^{1/2}\abs{c_1}^{1/2}} \\
    & = \frac{B}{\abs{a_1b_1c_1}} \left(\frac{B}{\abs{a_1b_1^2c_1^2}}\right)^{-1/4}
  \end{align*}
  for $V_1$. Since $\abs{a_1b_1^2c_1^2}\leq B$, Proposition~4.3 in~\cite{MR2520770} yields the desired result, for which are only left to check that the constant is indeed $\prod_p c_p$ with
  \begin{align*}
    c_p &= \frac{1}{p^2} \left(1-\frac{1}{p}\right) \left(1-\frac{1}{p}\right)^2 +
    2\frac{1}{p} \left(1-\frac{1}{p}\right)^2 \left(1-\frac{1}{p}\right) +
    \left(1-\frac{1}{p}\right)^2 \\
    & = \left(1-\frac{1}{p^2}\right)
    \left(1+\frac{2}{p}-\frac{1}{p^2}-\frac{1}{p^3} \right). \qedhere
  \end{align*}
\end{proof}

\begin{proof}[Proof of Theorem~\ref{thm:intro-count}]
  We are only left to provide an asymptotic expansion of $V(B)$. The error we introduce when removing the condition $\abs{a_1}\geq 1$ in the integral, while keeping the condition $\max\{\abs{b_0^2},\abs{b_1^2}\}\max\{\abs{c_0^2},\abs{c_1^2}\}\leq B$ implied by the others, is at most
  \begin{equation*}
    2\int_{\substack{\abs{c_0},\abs{c_1} \geq 1,\\ \max\{\abs{b_0^2},\abs{b_1^2}\}\max\{\abs{c_0^2},\abs{c_1^2}\}\leq B}} \diff b_0 \diff b_1 \diff c_0 \diff c_1
    \ \ll\ 
    \int_{\abs{c_0},\abs{c_1}\geq 1} \frac{B}{\max\{\abs{c_0^2},\abs{c_1^2}\}} \diff c_0 \diff c_1
    \ \ll\ 
     B\log B\text{.}
  \end{equation*}
  Using the symmetry in the integral, we get
  \[
    V(B)= \int_{\substack{\abs{b_0},\abs{b_1}, \abs{c_0}, \abs{c_1} \geq 1 \\ \abs{b_0}\leq \abs{b_1}, \abs{c_0} \leq \abs{c_1}, \\ \abs{b_1^2c_1^2}\leq B }} \frac{B}{\abs{b_1^2}\abs{c_1^2}} \diff b_0 \diff b_1 \diff c_0 \diff c_1 + O(B\log B) \text{.}
  \]
  Removing $\abs{b_0}\geq 1$ introduces an error of at most
  \[
    2 \int_{\substack{\abs{b_1},\abs{c_1}\geq 1\\\abs{c_0}\leq \abs{c_1}\leq \sqrt{B}/\abs{b_1}}}
    \frac{B}{\abs{b_1^2}\abs{c_1^2}} \diff b_1 \diff c_0 \diff c_1
    \ll \int_{1\leq\abs{c_1}\leq B} \frac{B}{\abs{c_1}} \diff c_1
    \ll B \log B\text{,}
  \]
  as, analogously, does removing $\abs{c_0}\geq 1$.
  We thus arrive at
  \begin{align*}
    V(B)
    &=4\int_{\substack{\abs{b_1},\abs{c_1}\geq 1,\\\abs{b_1}\leq \sqrt{B}/\abs{c_1}}} \frac{B}{\abs{b_1}\abs{c_1}} \diff b_1 \diff c_1 + O(B\log B)\\
    &=4 \int_{1\leq \abs{c_1} \leq \sqrt{B}} \frac{B\log B}{\abs{c_1}} + O(B\log B)
    =4 B (\log B)^2 + O(B\log B)\text{.}\qedhere
  \end{align*}
\end{proof}

\subsection{Interpretation of the result}\label{ssec:toric-interpretation}

As $E_1$ and $E_2$ intersect and this intersection has real points, while $M$ meets neither of the exceptional divisors, the analytic Clemens complex of $D$ consists of a $1$-simplex $A=\{E_1, E_2\}$ and an isolated vertex $\{M\}$ (which we will also simply denote by $M$). Integral points tend to accumulate around the boundary divisor; their number is dominated by those points lying near the intersection of a maximal number of boundary components. It is for this reason that the dimension of the Clemens complex is part of the exponent in the main theorem of~\cite{arXiv:1006.3345}.

\begin{figure}[ht]
  \begin{center}
    \begin{tikzpicture}
      \node (H) at (-2,0) [anchor=south]{$M$};
      \node (E1) at (0,0) [anchor=south]{$E_1$};
      \node (E2) at (2,0) [anchor=south]{$E_2$};

      \node (H) at (-2,0) [draw, shape=circle, fill=black, scale=.3]{};
      \node (E1) at (0,0) [draw, shape=circle, fill=black, scale=.3]{};
      \node (E2) at (2,0) [draw, shape=circle, fill=black, scale=.3]{};

      \draw[thick] (E1) -- (E2) node[pos=0.5, anchor=north]{$A$};
    \end{tikzpicture}
    \caption{The Clemens complex of $D$.}
  \end{center}
\end{figure}
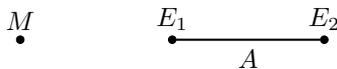

For the toric variety $X$, this does not hold. There is an obstruction to the existence of points near the intersection $E_1\cap E_2$ (and even to the existence of integral points near $E_1\cup E_2$):
Let us consider the rational function $f=a_1xy/a_0$ (in fact, a character of $T$) on $X$. It is a nonconstant regular function on $U_A=X\setminus M$, so there is an obstruction in the sense of Corollary~\ref{cor:obstruction}.

Concretely, this means the following: The function $f$ is a regular in a neighborhood of $E_1\cap E_2$, vanishing on $E_1\cap E_2$. If a point $p$ is near $E_1\cap E_2$, $\abs{f(p)}$ should thus be small. However, since $f$ is a regular function on $\mU$, its value is an integer at any integral point---and thus $\abs{f(p)}\ge 1$ except for points on the subvariety $\{f=0\}$. This means that the only integral points that are close to $E_1\cap E_2$ can be points on this subvariety (which we excluded in our counting problem).
For this reason we cannot expect a contribution of the maximal face $A$ of the Clemens complex to our asymptotic formula.
Since $f$ is even regular on neigborhoods of both $E_1$ and $E_2$, there can in fact be no integral points near either of those divisors, and we cannot expect a contribution of those two nonmaximal faces (a general phenomenon by Remark~\ref{rem:nonmaximal}).
The existence of this function also has an effect on the Picard group. That $f$ vanishes on $E_1$, $E_2$, and $M'=V(a_1)$, and that it has a pole on $M$ means that we have $[E_1]+[E_2]+[M']=[M]$ in $\Pic(X)$, and thus $[E_1]+[E_2]+[M']=0$ in $\Pic(X\setminus M)$. All three classes are nontrivial, hence the effective cone of $X\setminus M$ contains a plane. It is thus not strictly convex, and its characteristic function is identically $0$.

Since a value of the characteristic function is a factor of the leading constant in op.\ cit., this means that, for this variety, the leading constant is zero, contrary to their claim in Lemma~3.11.4.
In particular, this variety is an example for the obstruction in~\ref{ssec:obstruction}, and, more precisely, the situation considered in Theorem~\ref{thm:obstruction-construction}~\eqref{enum:nonconvex}.
The exponent of $\log B$ in Proposition~\ref{thm:intro-count} is one less than the one given by Chambert-Loir and Tschinkel. We can however interpret our asymptotic formula analogously to the formula given by Chambert-Loir and Tschinkel: There is no obstruction at the only remaining maximal face $M$ of the Clemens complex. Substituting this face for the maximal dimensional face $A$ of the Clemens complex, we get the correct asymptotic formula.
Summarizing, the situation is as follows:

\begin{proposition}\label{prop:toric-interpretation}
  \leavevmode

  \begin{enumerate}
    \item The cone $\Eff_A=\overline{\Eff}_{X\setminus M}\subset\Pic(X\setminus M)_\RR$, associated with the unique maximal-dimensional face $A$ of the Clemens complex, is not strictly convex.

    \item The cone $\Eff_M$, associated with the unique other maximal face $M$, is strictly convex.
    The constant associated with this face is $\alpha_M=1/8$, and the exponent associated with it is $b_M=b^\prime_M=3$.
  \end{enumerate}
\end{proposition}
\begin{proof}
  The Picard group $\Pic(U;A)=\Pic(X\setminus M)$ is the quotient
  \[
    \Pic(X)/[M] \cong \ZZ^5/\langle (1,0,0,0,0) \rangle \cong \ZZ^4.
  \]
  The effective cone is generated by the classes of the torus-invariant prime divisors
  \begin{align*}
    &(0,0,-1,-1),\ (1,0,0,0),\ (1,0,-1,0),\ (0,1,0,0),\\
    &(0,1,0,-1),\ (0,0,1,0),\ \text{and } (0,0,0,1),
  \end{align*}
  and thus contains the plane $\{(0,0,x,y)\mid x,y\in \RR\}$; in particular, it is not strictly convex.

  The Picard group $\Pic(U;H)=\Pic(U_M)$ for $U_M=X\setminus (E_1\cup E_2)$ is the quotient
  \[
    \Pic(X)/\langle[E_1],[E_2]\rangle \cong \ZZ^5/\left\langle(0, 0, 0, 1, 0),(0, 0, 0, 0, 1)\right\rangle
    \cong \ZZ^3\text{.}
  \]
  Its rank is $b_M^\prime=3$, so it coincides with
  \[
    b_M=\rk\Pic(U) - \rk E(U) + \# \{M\} = 2 - 0 + 1.
  \]
  The effective cone $\Eff_M=\Eff_{U_M}$ is smooth and generated by
  \[
    (1,0,0),\ (0,1,0),\ \text{and } (0,0,1).
  \]
  The image of the log anticanonical class in this quotient is $(1,2,2)$. The characteristic function of~$\Eff_M$ thus evaluates to $1/4$, and 
  \[
    \alpha_M=\frac{1}{(b_M^\prime -1)!} \frac{1}{4} = \frac{1}{8}.\qedhere
  \]
\end{proof}

\begin{remark}\label{rmk:gap}
  This exemplifies a gap in the proof of~\cite[Lem.~3.11.4]{arXiv:1006.3345} of which the authors were already aware and because of which they no longer believed in the correctness of the final result of their preprint: they do not check that the characteristic function $\mathscr{X}_{\Lambda^\prime_A}(\pi(\tilde{\lambda}))$ (in the notation of op.~cit., equal to $\charfun_{\Eff_A}(\pi (\omega_X(D)^\vee)) = (b^\prime_A-1)! \alpha_A$ in our notation) is nonzero, and this example demonstrates that it can be zero, making the leading constant 0. Note that the formula in op.~cit.\ is still correct if interpreted as $N(B)=0 \cdot B(\log B)^3+O(B(\log B)^{2})$, that is, as an upper bound. Here, we prove an asymptotic formula, of the form $N(B) \sim c B(\log B)^{2}$.
\end{remark}

To finish the proof of Theorem~\ref{thm:interpreted-count}, we are only left to compute the relevant Tamagawa volumes. To this end, consider the chart
\begin{align*}
  X \setminus V(a_1b_1c_1xy) \quad & \to \quad \qquad \bAA^3,\\
  (a_0:a_1:b_0:b_1:c_0:c_1:x:y) & \mapsto \left(\frac{a_0}{a_1xy},\frac{b_0}{b_1x},\frac{c_0}{c_1y}\right)
\end{align*}
and its inverse $\bAA^3\to X$
\[
  (a,b,c)\mapsto (a:1:b:1:c:1:1:1).
\]

\begin{lemma}
  Under this chart, the integral points $\mU(\ZZ_p)$ correspond to
  \[
    \{(a,b,c)\in\ZZ_p^3 \mid \text{either }\abs{a}= 1, \text{ or } \abs{a}>1 \text{ and } \abs{b},\abs{c} \le 1\}\text{.}
  \]
\end{lemma}
\begin{proof}
  Analogously to~\eqref{eq:int-point-on-torsor}, $\ZZ_p$-integral points are the images of $(1,a_1,b_0:,b_1,c_0,c_1,1,1)\in Y(\QQ_p)$  for $a_1,\dots,c_1\in\ZZ_p$ such that~\eqref{eq:gcd} holds.
  In particular, $|a|=|1/a_1| \ge 1$. If $|a|> 1$, then $|a_1|<1$;
  the coprimality conditions then imply $b_1c_1\in\ZZ_p^\times$, and thus $|b|=|b_0|,|c|=|c_0| \le 1$.

  On the other hand, let $(a,b,c)$ be a point in the above set. If $|a|=1$, let $a_1$ and $a_0=a^{-1}$. If $|b|\le 1$, let $b_0=b$ and $b_1=1$, else, let $b_0=1$ and $b_1=b^{-1}$, and set $c_0$, $c_1$ analogously. 
  Finally, if $|a|>1$, let $a_0=a$, $b_0=b$, $c_0=c$, and the remaining coordinates be $1$. In each case, $((a_0:a_1),(b_0:b_1),(c_0:c_1))$ is an integral point that maps to $(a,b,c)$.
\end{proof}

\begin{lemma}\label{lem:toric-volumes}
  We have
  \[
    \tau_{M,\infty}(M(\RR))=16
    \quad \text{and} \quad
    \tau_{U,p}(\mU(\ZZ_p))=1+\frac{2}{p}-\frac{1}{p^2}-\frac{1}{p^3}
  \]
  for all primes $p$.
\end{lemma}
\begin{proof}
  In order to get a metric on the canonical bundle inducing the Tamagawa measures, consider the isomorphism from $\omega_X$ to the bundle $\cL_{\omega_X}$ whose sections are elements of degree $\omega_X$ in the Cox ring that maps $\diff a \wedge \diff b \wedge \diff c$ to $a_1^{-2}b_1^{-2}c_1^{-2}x^{-1}y^{-1}$; then pull back the metric along this isomorphism.
  For the archimedean volume, we want to integrate
  \[
    \norm{1_{E_1} 1_{E_2} \diff b \wedge \diff c}_{\omega_M(E_1+E_2)}
    = \norm{a^{-1} 1_{E_1} 1_{E_2} \diff a \wedge \diff b \wedge \diff c}_{\omega_X(D)}
  \]
  over $M(\RR)$
  (regarding $a^{-1}$ as an element in $\Gamma(U,\cO(-M))\subset \Gamma(U,\mathcal{K}_{\bAA^3})$). Outside $M$, we have $a^{-1}=a^{-1} 1_{M}$, where the first factor is a section in $\Gamma(\bAA^3\setminus M,\cO_{\bAA^3})$, and thus
  \[
    \norm{1_{E_1} 1_{E_2} \diff b \wedge \diff c}_{\omega_M(E_1+E_2)}
    = \lim_{a\to 0} \left(
    \abs{a^{-1}}
    \norm{1_M 1_{E_1} 1_{E_2} \diff a \wedge \diff b \wedge \diff c}_{\omega_X(D)}\right).
  \]
  The norm $\norm{1_M 1_{E_1} 1_{E_2} \diff a \wedge \diff b \wedge \diff c}_{\omega_X(D)}$ is
  \[
    \frac{
      \max\{\abs{a_0},\abs{a_1xy}\}
    }{
      \abs{a_0xy}\max\{
        \abs{a_1^2b_0^2c_0^2},
        \abs{a_1^2b_1^2c_0^2x^2},
        \abs{a_1^2b_0^2c_1^2y^2},
        \abs{a_1^2b_1^2c_1^2x^2y^2},
        \abs{a_0^2b_1^2c_1^2}
      \}
    }
  \]
  at a point $\pc{a_0:a_1:b_0:b_1:c_0:c_1:x:y}\in X(\QQ)$ given in Cox coordinates. Evaluating in the image of a point $(a,b,c)$ yields
  \[
    \norm{\diff b \wedge \diff c}
    = \lim_{a\to 0}\frac{
      \abs{a}\max\{1,\abs{a}\}
    }{
      \abs{a}\max\{
      \abs{b^2c^2}, \abs{c^2}, \abs{b^2}, 1 , \abs{a^2}
    \}}
    = \frac{1}{\max\{1,\abs{b}^2\}\max\{1,\abs{c}^2\}}
  \]
  Integrating this results in the archimedean Tamagawa volume
  \[
    \tau_{(M,E_1+E_2),\infty} (M(\RR)) =
    \int_{\RR^2} \frac{1}{\max\{1,\abs{b}^2\}\max\{1,\abs{c}^2\}} \diff b \diff c = 16,
  \]
  which has to be renormalized with the factor $c_\RR=2$.

  In order to determine the Tamagawa volumes at the nonarchimedean places, we integrate
  \[
    \norm{1_M 1_{E_1} 1_{E_2} \diff a \wedge \diff b \wedge \diff c}_{\omega_X(D)}
  \]
  over $\mU(\ZZ_p)$.
  Using the same description as above, this results in
  \[
    \left(1-\frac{1}{p}\right)
    \int_{b,c\in\QQ_p}
    \frac{1}{\max\{1,\abs{b}^2\}\max\{1,\abs{c}^2\}} \diff b \diff c
    + \int_{\substack{\abs{a}>1\\ \abs{b},\abs{c}\le 1 }} \frac{1}{\abs{a}^2} \diff a\diff b \diff c.
  \]
  The first integral is
  \[
    \left(\int_{b\in\QQ_p} \frac{1}{\max\{1,\abs{b}^2\}} \diff b\right)^2
    = \left(1 + \int_{\abs{b}>1} \frac{1}{\abs{b}^2} \diff b\right)^2
    = \left(1+\frac{1}{p}\right)^2\text{,}
  \]
  and the second is
  \[
    \int_{\abs{a}>1} \frac{1}{\abs{a}^2} \diff a = \frac{1}{p} \text{,}
  \]
  so, in total, we get
  \[
    \tau_{U,p}(\mU(\ZZ_p)) = \left(1-\frac{1}{p}\right)\left(1+\frac{1}{p}\right)^2  + \frac{1}{p}= 1+\frac{2}{p} - \frac{1}{p^2} - \frac{1}{p^3}\text{.} \qedhere
  \]
\end{proof}

\begin{proof}[Proof of Theorem~\ref{thm:interpreted-count}]
  Comparing $\alpha_\uA$ as computed in Proposition~\ref{prop:toric-interpretation} and the descriptions of the Tamagawa volumes in Lemma~\ref{lem:toric-volumes} to the asymptotic formula in Theorem~\ref{thm:intro-count} finishes the proof.
\end{proof}

\bibliographystyle{alpha}

\bibliography{../bibliography/bibliography}

\end{document}